\documentclass[reqno,centertags,12pt]{amsart}
\usepackage{fullpage}
\usepackage{amsmath,amssymb,amscd}
\usepackage{color}
\usepackage{enumerate}
\usepackage{hyperref}
\usepackage{graphicx}
\usepackage{amsfonts}
\usepackage{fancyhdr}
\usepackage{indentfirst}

\newtheorem{theorem}{Theorem}[section]
\newtheorem{cor}[theorem]{Corollary}
\newtheorem{lemma}[theorem]{Lemma}

\newtheorem{prop}[theorem]{Proposition}

\theoremstyle{remark}
\newtheorem{remark}[theorem]{Remark}

\theoremstyle{definition}
\newtheorem{defin}[theorem]{Definition}

\newcommand{\Aa}{L^2_{x,t}}

\newcommand{\B}{X^{-s, 1+s}}
\newcommand{\DB}{X^{-s, s}}
\newcommand{\C}{X^{-3-4s, 2s+2}}
\newcommand{\DC}{X^{-3-4s, 2s+1}}
\newcommand{\K}{|D|^{-2s-2}X^{\frac{1}{4}, \frac{1}{4}}_{\tau=\frac{1}{4}\xi^3}}
\newcommand{\DK}{|\partial_t+\partial^3_x|^{-1}|D|^{-2s-2}X^{\frac{1}{4}, \frac{1}{4}}_{\tau=\frac{1}{4}\xi^3}}
\newcommand{\LP}{|D|L^\infty_{t,x}}
\newcommand{\DLP}{|D_t-D^3_x||D|L^\infty_{t,x}}
\newcommand{\DI}{|D|^{-s}|I|^{-\frac{1}{2}}L^2}
\newcommand{\Il}{|I_\lambda|}
\newcommand{\xll}{X^1[I_\lambda]}
\newcommand{\xaa}{X^1[I_\alpha]}
\newcommand{\xbb}{X^1[I_\beta]}
\newcommand{\xmm}{X^1[I_\mu]}

\newcommand{\zll}{Z[I_\lambda]}
\newcommand{\zaa}{Z[I_\alpha]}

\newcommand{\Imu}{|I_\mu|}
\newcommand{\ill}{I_\lambda}
\newcommand{\iaa}{I_\alpha}
\newcommand{\ibb}{I_\beta}

\newcommand{\imm}{I_\mu}
\newcommand{\threez}{-\frac{4}{3}(s+1)-\frac{1}{3} }
\newcommand{\sixz}{-\frac{2}{3}(s+1)+\frac{1}{3} }
\newcommand{\tail}{ \frac{2}{3}(s+1) }
\newcommand{\sm}{\sigma_m}
\newcommand{\sM}{\sigma_M}
\newcommand{\LL}{(\partial_t+\partial^3_{x})}
\newcommand{\ul}{u_\lambda}
\newcommand{\um}{u_\mu}
\newcommand{\xl}{X_{\lambda}}
\newcommand{\yl}{Y_{\lambda}}
\newcommand{\ua}{u_\alpha}
\newcommand{\vl}{v_\lambda}
\newcommand{\vm}{v_\mu}
\newcommand{\va}{v_\alpha}
\newcommand{\ub}{u_\beta}
\newcommand{\vb}{v_\beta}

\newcommand{\vr}{v_\gamma}
\newcommand{\Qm}{Q_{\sigma_m}}
\newcommand{\QM}{Q_{\gtrsim M}}
\newcommand{\pq}{P_\alpha Q_{\sigma}}

\newcommand{\rs}{\alpha^{2s+\frac{3}{2}}L^2_xL^{\infty}_t}
\newcommand{\elt}{\eta_{\lambda}(t)}
\newcommand{\eit}{\eta_{I}(t)}
\newcommand{\el}{\eta_\lambda}

\newcommand{\ea}{\eta_\alpha}

\newcommand{\ej}{\eta_J}
\newcommand{\emt}{\eta_\mu(t)}

\newcommand{\eb}{\eta_\beta}

\newcommand{\clx}{\chi^\lambda(x)}
\newcommand{\cljx}{\chi^\lambda_j(x)}

\newcommand{\cax}{\chi^\alpha(x)}
\newcommand{\txft}{L^2_xL^{\infty}_t}
\newcommand{\txt}{L^2_{t,x}}

\newcommand{\fxt}{L^{\infty}_{t,x}}
\newcommand{\fxtt}{L^{\infty}_xL^2_t}
\newcommand{\tr}{\tilde{R}}
\newcommand{\td}[1]{\tilde{#1}}
\newcommand{\ta}{{a}}

\newcommand{\lag}{\langle}
\newcommand{\rag}{\rangle}
\newcommand{\hu}{\hat{u}}
\newcommand{\ddt}{\frac{d}{dt}}
\newcommand{\beq}{\begin{eqnarray}}
\newcommand{\eeq}{\end{eqnarray}}
\newcommand{\ben}{\begin{eqnarray*}}
\newcommand{\een}{\end{eqnarray*}}
\newcommand{\xx}{X^s\cap X^s_{le}}
\newcommand{\ax}[1]{a(\xi_{#1})}
\newcommand{\xii}[1]{\xi_{#1}}
\newcommand{\lsm}{\lesssim}
\newcommand{\so}{-\frac{4}{5}}
\newcommand{\lh}{l^2_{\lambda}L^{\infty}_tH^s}

\newcommand{\leh}{l^{2}_{\lambda}l^{\infty}_j L^2_tH^{-s-\frac{3}{2}}}
\begin{document}


\title{A-priori bounds for  KdV equation below $H^{-\frac{3}{4}}$ }
\author{Baoping Liu}
\address{Department of Mathematics\\ University of California, Berkeley}
\email{baoping@math.berkeley.edu}
%
\subjclass[2000]{
35Q55. \newline\indent The author was supported in part by NSF
grant DMS0801261.}

\vspace{-0.3in}
\begin{abstract}
We consider the Korteweg-de Vries Equation (KdV) on the real line,
and prove that the smooth solutions satisfy a-priori local in time
$H^s$ bound in terms of the $H^s$ size of the initial data for
$s\geq\so$.
\end{abstract}

\maketitle \setcounter{tocdepth}{1}
\tableofcontents\section{Introduction} In this paper, we consider
the Korteweg-de Vries (KdV) equation,
\begin{equation}\left\{
\begin{array}{l} \partial_t u+\partial^3_x u+\partial_x (u^2)=0, \hspace{5mm} u:\mathbb{R}\times [0,T]\rightarrow \mathbb{R}, \\
u(0)=u_0 \in H^{s}(\mathbb{R}).
\end{array}\right.\label{kdv}\end{equation}
The equation is invariant respect to the scaling law
$$u(t,x)\rightarrow \lambda^2 u(\lambda^3 t,\lambda x),$$
which implies the scale invariance for initial data in
$\dot{H}^{-\frac{3}{2}}(\mathbb{R})$. It has been shown to be
locally well-posed (LWP) in $H^s$ for $s>-\frac{3}{4}$ by Kenig,
Ponce and Vega~\cite{KPV1} using a bilinear estimate. They
constructed solution on a time interval $[0,\delta]$, with
$\delta$ depending on $\|u_0\|_{H^s(\mathbb{R})}$. Later, the
result was extended to global well-posedness (GWP) for
$s>-\frac{3}{4}$ by Colliander, Keel, Staffilani, Takaoka and
Tao~\cite{CKSTT} using the I-method and almost conserved
quantities. See also the references \cite{BS}, \cite{Cohen},
\cite{Kato}, \cite{GTV}, \cite{KPV3}, \cite{Bou1}, \cite{KPV4} for
earlier results, and  \cite{CCT}, \cite{Guo}, \cite{Kishi} for
local and global results at the endpoint $s=-\frac{3}{4}$.

In ~{\cite{NTT}}, Nakanishi, Takaoka and Tsutsumi showed that the
essential bilinear estimate fails if $s<-\frac{3}{4}$. In fact,
Christ, Colliander and Tao~{\cite{CCT}} proved a weak form of
illposedness of the $\mathbb{R}$-valued KdV equation for
$s<-\frac{3}{4}$. Precisely, they showed that the solution map
fails to be uniformly continuous. See \cite{KPV2} for the
corresponding result for the $\mathbb{C}$-valued KdV equation.

On the other hand, the same question was posed in the periodic
setting ($u:\mathbb{T}\times [0,T]\rightarrow \mathbb{R}$), where
for $s\geq -1/2$, we have the results of LWP{\cite{KPV1}} and
GWP{\cite{CKSTT}}. Also, Kappeler and Topalov{~\cite{KaTo}}, using
the inverse scattering method~\cite{IST1}, proved GWP for inital
data in $H^\beta({\mathbb{T}}),\beta\geq -1$ in the sense that the
solution map is $C^0$ globally in time. Their proof depends
heavily  on the complete integrability of the KdV equation.
Interested readers are also referred to the work of Lax and
Levermore~\cite{Lax}, Deift  and Zhou~\cite{Deift1},
~\cite{Deift2}. There they used inverse scattering and
Riemann-Hilbert methods to study the semiclassical limit of the
completely integrable equations.

Concerning the KdV problem with initial data in
$H^{-1}(\mathbb{R})$, there has been several results recently. In
~\cite{Molinet1}, Molinet showed that the solution map can not be
continuously extended in $H^s{(\mathbb{R})}$ when $s<-1$.
In~\cite{KPST}, Kappeler, Perry, Shubin and Topalov showed that
given certain assumptions on the initial data $u_0\in H^{-1}$ ,
there exists a global weak solution to the KdV equation.
Buckmaster and Koch~\cite{BK} proved the existence of weak
solutions to KdV equation with $H^{-1}$ initial data. The approach
in ~\cite{KPST} and ~\cite{BK} both use the Miura transformation
to link the KdV equation to the mKdV equation, and the proofs
involve the  study of   Muria map, and the existence of weak $L^2$
solutions to mKdV or mKdV around a soliton.

In addition, there is an interesting result by Molinet and
Ribaud~\cite{MR1} on the initial-value problem for KdV-Burgers
equation.
\begin{equation}\left\{
\begin{array}{l} \partial_t u+\partial^3_x u+\partial_x (u^2)-\partial_x^2u=0, \hspace{5mm} t\in \mathbb{R}_{+},\hspace{2mm} x\in \mathbb{R}  \textnormal{  or  }  \mathbb{T},\\
u(0)=u_0 \in H^{s}(\mathbb{R}).
\end{array}\right.\label{kdvBurgers}\end{equation}
They showed that~(\ref{kdvBurgers}) is GWP in the space $H^s
(\mathbb{R})$ for $s \geq -1$, and ill-posed when $s <-1$ in the
sense that the corresponding solution map is not $C^2$. This is a
bit surprising since the initial-value problem for the Burgers
equation
\begin{equation}\left\{
\begin{array}{l} \partial_t u+\partial_x (u^2)-\partial_x^2u=0, \hspace{5mm} t\in \mathbb{R}_{+},\hspace{2mm} x\in \mathbb{R}  ,\\
u(0)=u_0 \in H^{s}(\mathbb{R}).
\end{array}\right.\label{Burgers}\end{equation}
is known to be LWP in the space $H^s (\mathbb{R})$ for $s \geq
-\frac{1}{2}$, and is ill-posed in $H^s (\mathbb{R})$ for $s <
-\frac{1}{2}$, see references~\cite{Bek} and ~\cite{Dix}. Notice
that the critical result for Burgers equation (\ref{Burgers})
agrees with prediction from usual scaling arguments. While
KdV-Burgers equation(\ref{kdvBurgers}) has no scaling invariance,
the sharp   result  by Molinet and Ribaud $s=-1$ is lower than
$s=-\frac{3}{4}$ for KdV, and $s=-\frac{1}{2}$ for Burgers
equation.

From all the  results mentioned before, it seems reasonable to
conjecture well-posedness  of KdV equation (\ref{kdv}) in
$H^s(\mathbb{R})$, in the range $-1\leq s<-\frac{3}{4}$, with some
continuous but not uniform continuous dependence on the initial
data.

Another related topic is  one dimensional cubic Nonlinear
Schr\"{o}dinger equation (NLS)
\begin{equation}\left\{
\begin{array}{l} i\partial_t u+\partial^2_x u\pm|u|^2u=0, \hspace{5mm}u:\mathbb{R}\times [0,T]\rightarrow \mathbb{C}, \\
u(0)=u_0 \in H^{s}(\mathbb{R}).
\end{array}\right.\end{equation}
The NLS has scaling invariance for initial data in
$\dot{H}^{-\frac{1}{2}}(\mathbb{R})$. It has GWP for initial data
in $u_0\in L^2$ and locally in time the solution has a uniform
Lipschitz dependence on the initial data in balls. But below this
scale, it has been shown that uniform dependence
fails~{\cite{CCT}},{\cite{KPV2}}. Koch and Tataru~\cite{KT2}
proved an a-priori local-in-time bounds for initial data in
$H^{s}, s\geq -\frac{1}{4}.$ Similar results were previously
obtained by Koch and Tataru~\cite{KT1} for $s\geq-\frac{1}{6}$,
and  by Colliander, Christ and Tao \cite{CCT2} for $s>
-\frac{1}{12}$. These a-priori estimates ensure that the equation
is satisfied in the sense of distributions even for weak limits,
and hence they also obtain existence of global weak solutions
without uniqueness.

Inspired by the results above, we look at the KdV equation with
initial data in $H^{s}$ when $s<-\frac{3}{4}$,  and prove that the
solution satisfies a-prori local in time $H^{s}$ bounds in terms
of the $H^{s}$ size of the initial data, for $s\geq \so$. The
advantage here is that we
 performed detailed analysis about the
interactions in the nonlinearity, which gives us better
understanding of the real obstruction towards establishing
wellposedness result in low regularity.

Our main result is as follows:
\begin{theorem}\label{aprioribound}
(A-priori bound) Let $s\geq \so$. For any $M
> 0$ there exists time $T$ and constant $C$, so that for any
initial data in $H^{-\frac{3}{4}}$ satisfying
\[\|u_0\|_{H^{s}}<M,\] there exists a solution $u \in
C([0,T],H^{-\frac{3}{4}})$ to the KdV equation which satisfies
\beq\label{unibound}\|u\|_{L^{\infty}_tH^{s}}\leq C
\|u_0\|_{H^{s}}.\eeq
\end{theorem}
Using the uniform bound (\ref{unibound}), together with the
uniform bound on nonlinearity
\[\|\chi_{[-T,T]}u\|_{\xx}+\|\chi_{[-T,T]}\partial_x(u^2)\|_{\xx}\lsm
\|u_0\|_{H^s},\] which come as a byproduct of our analysis in the
previous theorem, one may also prove the existence of weak
solution following a similar argument as in \cite{CCT2}.

\begin{theorem}
(Existence of weak solution) Let $s\geq \so$. For any $M
> 0$ there exists time $T$ and constant $C$, so that for any
initial data in $H^{s}$ satisfying
\[\|u_0\|_{H^{s}}<M,\] there exists a weak solution $u \in
C([0,T],H^s)\cap (\xx)$ to the KdV equation which satisfies
\[\|u\|_{L^{\infty}_tH^{s}}+\|\chi_{[-T,T]}u\|_{\xx}+\|\chi_{[-T,T]}\partial_x(u^2)\|_{\xx}\leq C \|u_0\|_{H^{s}}.\]
\end{theorem}

\begin{remark}{We can always rescale the initial data and hence
just need to prove the theorems in case $M\ll 1$.
 }\end{remark}
We begin    with a Littlewood-Paley   frequency decomposition of
the solution $u$,
\[u=\sum_{\lambda\geq 1, \text{ dyadic}}u_{\lambda}\]
Here we put all frequencies smaller than 1 into one piece.

For each $\lambda$ we also use a spatial partition of unity on the
$\lambda^{4s+5}$ scale
\[1=\sum_{j\in\mathbb{Z}}\chi_j^\lambda(x),\hspace{5mm} \chi_j^\lambda(x)=\chi(\lambda^{-4s-5}x-j),\]
with $\chi(x)\in C_0^\infty(-1,1)$.

In order to prove the theorem, we need Banach spaces
 \begin{itemize}
    \item $X^s$ and $X^s_{le}$ to measure the
regularity of the solution $u$. The first one measures dyadic
pieces of the solution on a frequency dependent timescale, and the
second one measures the spatially localized size of the solution
on unit time scale. They are similar to the ones used by Koch and
Tataru in~\cite{KT2}.
    \item The corresponding $Y^s$ and $Y^s_{le}$ to measure the regularity of the
    nonlinear term.
    \item Energy spaces
    \[\|u\|^2_{l^2_{\lambda}L^{\infty}_tH^s}=\sum_{\lambda\geq 1}\lambda^{2s}\|u_{\lambda}\|^2_{L^{\infty}_tL^2_x},\]
    and a local energy space
    \[\|u\|^2_{l^{2}_{\lambda}l^{\infty}_j L^2_tH^{-s-\frac{3}{2}}}=\sum_{\lambda\geq 1}\lambda^{-2s-5}\sup_j\|\chi_j^{\lambda}\partial_x u_{\lambda}\|^2_{L^2_{x,t}}.\]
\end{itemize}

With the spaces above, we will prove the following three
propositions.

The first one is about the linear equation.
\begin{prop} \label{linearprop}The following energy estimates hold
for \textnormal{(\ref{kdv}):} \beq\|u\|_{X^s}\lesssim
\|u\|_{l^2_{\lambda}L^{\infty}_tH^s}+\|\LL u\|_{Y^s},\eeq
\beq\|u\|_{X_{le}^s}\lesssim \|u\|_{l^{2}_{\lambda}l^{\infty}_j
L^2_tH^{-s-\frac{3}{2}}}+\|\LL u\|_{Y_{le}^s}.\eeq
\end{prop}

The second one controls the nonlinearity.
\begin{prop}\label{bilinearprop}
Let $s> -1$ and  $u\in \xx$ be  a solution to equation
\textnormal{(\ref{kdv})}, then \beq\|\partial_x(u^2)\|_{Y^s\cap
Y^s_{le}}\lesssim \|u\|^2_{X^s\cap X^s_{le}}+\|u\|^3_{X^s\cap
X^s_{le}}.\eeq
\end{prop}

Finally, to close the argument we need to propagate the energy
norms. \begin{prop} \label{energypropagation}Let $s\geq \so$ and
$u$ be a solution to the \textnormal{(\ref{kdv})} with
$$\|u\|_{l^2_{\lambda}L^{\infty}_tH^s}\ll 1.$$ Then we have the bound
for energy norm
\beq\label{energybound}\|u\|_{l^2_{\lambda}L^{\infty}_tH^s}\lesssim
\|u_0\|_{H^s}+\sum_{k=3}^6\|u\|^k_{X^s\cap X^s_{le}},\eeq and
respectively the local energy norm
\beq\label{localenergybound}\|u\|_{\leh} \lesssim
\|u_0\|_{H^s}+\sum_{k=3}^6\|u\|^k_{X^s\cap X^s_{le}}.\eeq
\end{prop}
We organize our paper as follows: In section~\ref{sec2}, we will
define the spaces $X^s, X^s_{le}$, respectively $Y^s, Y^s_{le}$,
and  establish the linear mapping properties in
Proposition~\ref{linearprop}. In section \ref{sec3} we discuss the
linear and bilinear Strichartz estimates for free solutions, and
collect some useful estimates related to our spaces. In
section~\ref{sec4} we control the nonlinearity as in Proposition
\ref{bilinearprop}. In  sections \ref{sec5}, \ref{sec6} we use a
variation of the I-method to construct a quasi-conserved energy
functional and compute its behavior along the flow, thus proving
Proposition~\ref{energypropagation}.

Now we end this section by showing that the three propositions
imply Theorem~\ref{aprioribound}.
\begin{proof}Since $u_0\in H^{-\frac{3}{4}}$,
we  can solve the equation iteratively to get a solution up to
time 1, which implies that $u\in \lh$ and also that $u\in\xx$,
because the space we use has the nesting property $X^{s_1}\subset
X^{s_2}, s_1<s_2$, same for $\lh$ and $X^s_{le}$.

Then we use a continuity  argument. Suppose $\epsilon$ is a small
constant and $\|u_0\|_{H^s{(\mathbb{R})}}< \epsilon$. Take a small
$\delta$, so that $\epsilon\ll\delta\ll 1$, denote
\[A=\{ T\in[0,1]; \hspace{1mm}\|u\|_{\lh([0,T]\times\mathbb{R})} \leq2\delta,\hspace{3mm}
\|u\|_{\xx([0,T]\times\mathbb{R})}\leq 2\delta\}\] and we just
need to prove $A=[0,1]$. Clearly A is not empty and $0\in A$. We
need to prove that it is closed and open.

From definition in the next section, we can see that the norms
used in $A$ are continuous with respect to $T$, so A is closed.

Secondly, if $T\in A$, we have by proposition
\ref{energypropagation}
\[\|u\|_{\lh([0,T]\times\mathbb{R})}\lsm \epsilon +\delta^3,\]
and by proposition \ref{linearprop} and \ref{bilinearprop}, we
have
\[\|u\|_{\xx([0,T]\times\mathbb{R})}\lsm \epsilon+\delta^2+\delta^3.\]
So by taking $\epsilon $ and $\delta$ sufficiently small, we can
conclude that
\[\|u\|_{\lh([0,T]\times\mathbb{R})} \leq \delta,\hspace{2mm}
\|u\|_{\xx([0,T]\times\mathbb{R})}\leq  \delta.\] Since the norms
are continuous with respect to  $T$, it follows that a
neighborhood of $T$ is in $A$. Hence we proved
Theorem~\ref{aprioribound}.
\end{proof}

\section {Function spaces}
\label{sec2} The idea here follows the work of Koch and
Tataru~\cite{KT1}\cite{KT2}. We begin with some heuristic
argument: If the initial data in (\ref{kdv}) has norm
$\|u_0\|_{H^{-\frac{3}{4}}}\leq 1$, then the equation can be
solved iteratively up to time 1. Now when taking the same problem
with initial data $u_0\in H^{s}, s<-\frac{3}{4}$, localized at
frequency $\lambda$, the initial data will have norm
$\|u\|_{H^{-\frac{3}{4}}}\leq \lambda^{-s-\frac{3}{4}} $. Now if
we rescale it  to have $H^{-\frac{3}{4}}$ norm 1, we  see that the
evolution will still be described by linear dynamics on time
intervals of size $\lambda^{4s+3} $. 
So we decompose our solution into frequency pieces
$u=\sum_{\lambda\geq 1}u_{\lambda}$ and measure each piece
uniformly in size $\lambda^{4s+3} $ time intervals.

Another important idea is to look at waves of frequency $\lambda$
travelling with speed $\lambda^2$, so for time $\lambda^{4s+3} $,
it travels in spatial region of size $\lambda^{4s+5}$. So we also
decompose the space into a grid of size $\lambda^{4s+5}$ by using
the partition of unity
\[1=\sum_{j\in \mathbb{Z}}\chi_j^{\lambda}(x).\]
$\chi_j^{\lambda}(x)$ is defined as before, and it's easy to see
that the spatial scales increase with $\lambda$.

Bourgain's $X^{s,b}$ spaces are defined by
\[\|u\|^2_{X^{s,b}_{\tau=\xi^3}}=\int |\hat{u}(\tau,\xi)|^2(1+|\xi|
^2)^s(1+|\tau-\xi^3|^2)^b d\tau d\xi.\] We will use a modified
version of it on frequency or modulation dyadic pieces.

We start with spatial Littlewood-Paley decomposition,
\[u=\sum_{\lambda\geq 1 ~ \textnormal{dyadic} }P_{\lambda}u=\sum_{\lambda\geq 1
~ \textnormal{dyadic} }u_{\lambda}.\]

Also, we will use  the decomposition with respect to modulation
$|\tau-\xi^3|$
\[1=\sum_{\lambda\geq 1 ~ \textnormal{dyadic} }Q_{\lambda}.\] Both
decompositions are inhomogeneous, and uniformly bounded on
$X^{s,b}$ spaces.

Denote $\eta_I{(t)}$ as sharp time cutoff with respect to any time
interval $I$. Let $I_\lambda$   be a time interval of size
$\lambda^{4s+3}$, then we  use
 $\elt$ or $\el$  as a simplified notation for $\eta_{I_\lambda}(t)$.
And  $\chi^\lambda(x)$ is the smooth space cutoff with respect to
spatial intervals of size $\lambda^{4s+5}$ as before.

Define $|D|^{\alpha}$  to be the multiplier operator with Fourier
multiplier $|\xi|^{\alpha}$. We use the convention that $f\in
|D|^{-s} X \Leftrightarrow
\|f\|^2=\sum\lambda^{2s}\|f_\lambda\|^2_X<\infty$ in our
definitions.
\begin{defin}The spaces we use contain the following elements:
\begin{enumerate}[(i)]
\item Given an interval $I=[t_0,t_1]$, we define the space
\[\|\phi\|^2_{X^{0,1}[I]}=\|\phi(t_0)\|_{L^2}^2+|I|\|(\partial_t+\partial^3_{x})\phi\|_{L^2[I]}^2,\]
\[\|\phi\|^2_{X^1[I]}=\sum_{\lambda}\lambda^{2s}\|\phi_{\lambda}\|_{X^{0,1}[I]}^2.\]
$X^1[I]$ is used to control the low modulation part of the
solution in a classical space, which is extendable on the real
line. \item We use sums of spaces, i.e.
$\|u\|_{A+B}=\inf\{\|u_1\|_A+\|u_2\|_B, u=u_1+u_2\}$ to define
\[Z=(X^{-3-4s,2s+2}_{\tau=\xi^3}+\K)\cap \LP.\]
 $Z$
will always be used for very high modulations ($\geq |\xi|^3$),
i.e. in what are called the elliptic region.

\item The space $S$ is defined by putting high and low modulation
in different spaces.
\[\|u_{\lambda}\|_{S}=\lambda^{3s+\frac{3}{2s}+\frac{11}{2}}\|Q_{\sigma\leq\lambda^{4+\frac{3}{2s}} }\ul\|_{\Aa}+\|Q_{\lambda^{4+\frac{3}{2s}}\leq\sigma \leq\frac{1}{10}
\lambda^3}u_{\lambda}\|_{X^{-s,1+s}_{\tau=\xi^3}}+\|Q_{\geq\frac{1}{10}
\lambda^3}u_{\lambda}\|_{Z}.\] The good thing here is space $S$ is
stable with respect to sharp time truncations, the $L^2$ structure
deals with the tails when multiplying by a time-interval cutoff.
\\ In particular, we have
\[\|\elt\ul\|_{S}\lesssim\|\ul\|_{S}.\]

\item Let $X_\lambda[I]=X^1[I]+S[I]$. Now we can define $X^s$ norm
in a time interval $I$ by measuring the dyadic parts of $u$ on
small frequency-dependent time scales
\[\|u\|^2_{X^s[I]}=
\sum_{\lambda\geq 1}\sup_{|J|=\lambda^{4s+3}, J\subset I} \|
\eta_{J}(t)u_\lambda\|_{X_{\lambda}[J]}^2,\] $X^s_{le}$ measures
the spatially localized size of the solution on the unit time
scale
\[\|u\|_{X^s_{le}[I]}^2=\sum_{\lambda\geq 1}\sup_j\sum_{|J|=\lambda^{4s+3}, J\subset
I}\|\chi^{\lambda}_j(x)\eta_{J}(t)u_\lambda\|_{X_{\lambda}[J]}^2.\]
\item Correspondingly, we have the space $Y^s$ and $Y^s_{le}$
\[\|u\|^2_{Y^s[I]}=
\sum_{\lambda\geq 1}\sup_{|J|=\lambda^{4s+3}, J\subset I} \|
\eta_{J}(t)u_\lambda\|_{Y_{\lambda}[J]}^2,\]
\[\|u\|_{Y^s_{le}[I]}^2=\sum_{\lambda\geq 1}\sup_j\sum_{|J|=\lambda^{4s+3}, J\subset
I}\|\chi^{\lambda}_j(x)\eta_{J}(t)u_\lambda\|_{Y_{\lambda}[J]}^2.\]
Here \[Y_\lambda[I]=|D_x|^{-s}|I|^{-\frac{1}{2}}L^2+DS[I],\] $DS
=\{f= (\partial_t+\partial^3_{x})u;\hspace{1mm} u\in S \}$ with
the induced norm and $DS[I]=\{f|_{I}, f\in DS\}$.

Through our paper, we will mostly drop the interval $I$ in the
notation if $I=[0,1]$.
\end{enumerate}
\end{defin}

\begin{remark} \label{functionspaceremark}We look at each of the spaces in detail.
\begin{enumerate}
\item \label{remarkonx} $X^1[I]$ is not stable with respect to
sharp time truncation as it would cause jumps at both ends. Also
in order to talk about modulation, we need to extend functions so
that they are defined on the real line. To fix the problem, we
define
\[\|\phi\|^2_{X^{0,1}_I}=\|\phi(t_0)\|_{L^2}^2+|I|\|(\partial_t+\partial^3_{x})\phi\|_{\txt}^2,\]
\[\|\phi\|^2_{X^1_I}=\sum_{\lambda}\lambda^{2s}\|\phi_{\lambda}\|_{X^{0,1}_I}^2.\]
Now take any function  $u\in X^1[I]$, denote
$u_{E}=\theta(t)\widetilde{u}$, where $\widetilde{u}$ is the
extension of $u$ by free solutions with matching data at both ends
and $\theta(t)$ is a smooth cutoff on a neighborhood of $I$.
Clearly, $\|u\|_{X^1[I]}=\|u_E\|_{X^1_I}$, and when we talk about
function   $u\in X^1[I]$, we always mean $u_E$.

While $S[I]$ is stable with sharp time cutoff, $DS[I]$ is not. We
can extend functions in $S[I]$ by 0 outside the interval. And from
the definition, functions in  $DS[I]$ always come from  interval
restriction of functions in $DS$, which are defined on the real
line. \item The space $X^1[I]$ is compatible with  solutions to
the homogeneous equation. Namely for any smooth time cutoff
$\eta(t)$, we can prove
\[\|\eta (t)e^{t\partial^3_{x}}u_{0}\|_{X^1[I]}\lesssim \|u_{0}\|_{H^s},\]
It is also  compatible with energy estimates
\[\|u\|_{L^{\infty}_t(I;H^s)}\lesssim \|u\|_{X^1[I]}.\]


\item In our paper,  we will ignore the subscript notation
$\tau=\xi^3$ in the $X^{s,b}_{\tau=\xi^3}$ space except for the
special curve $\tau=\frac{1}{4}\xi^3$ which arises when two high
frequency wave interact and generate an almost equally high
frequency. \item \label{modulationmatch}Since we are using sums of
spaces, it is interesting to compare the norms of these spaces. We
note the following facts by Bernstein inequality.
\[\left\{
\begin{array}{ll}
    \|\ul\|_{X_{\lambda}[I_\lambda]}\approx \|\ul\|_{X^1[I_\lambda]}, & \mbox{when $|\tau-\xi^3|\lesssim \lambda^{4+\frac{3}{2s}}$,} \\
    \|\ul\|_{Z}\approx \|\ul\|_{\K \cap \LP}, & \mbox{when $|\tau-\frac{1}{4}\xi^3|\leq \frac{1}{10}\lambda^3 $,} \\
    \|\ul\|_{Z}\approx \|\ul\|_{\C\cap\LP}\approx\|\ul\|_{\B}, & \mbox{when $|\tau-\xi^3|\approx\frac{1}{10}\lambda^3 $.}
\end{array}
\right.\]

The $X^1$ and $S$ norm balance at modulation $|\tau-\xi^3|\approx
\lambda^{4+\frac{3}{2s}}$, which is also where we split $S$ into
the $L^2$ structure and $\B$. Hence whenever we split  into an
$X^1$ and an $S$ part, we always assume the $S$ part have
modulation larger than $\lambda^{4+\frac{3}{2s}}$ (which is larger
than $\lambda^2$). The same applies for $\DI$ and $DS$.

The third equality is because   when modulation is around
$\frac{1}{10}\lambda^3$, the $Z$ norm is in fact $\C\cap\LP$.
Using Bernstein, we can see that it matches with $\B$.

\end{enumerate}

\end{remark}
Now let us prove Proposition \ref{linearprop}.
\begin{proof}  It suffices to prove the Proposition for a fixed dyadic frequency $\lambda$.
We restrict our attention to time interval $J=[a,b]$ with size
$\lambda^{4s+3}$,  and we need to prove that
\beq\label{proveenergyestimate}\|\ul\|_{\xl[J]}\lesssim
\|\ul\|_{L^\infty_tH^s}+\|f_\lambda\|_{\yl[J]}, \hspace{0.5in} \LL
\ul =f_\lambda.\eeq We now split $f_\lambda$ into two components
\[f_\lambda=f_{1,\lambda}+f_{2,\lambda}, \hspace{0.2in}f_{1,\lambda}\in L^2, \hspace{0.2in}f_{2,\lambda}\in DS.\] Pick $\vl$ such that $\LL \vl=f_{2,\lambda},
\|f_{2,\lambda}\|_{DS}=\|\vl\|_S$. (or
$({v^i_{\lambda}})_1^{\infty}$ with
$\|v^i_{\lambda}\|_S\rightarrow \|f_2\|_{DS}$.)

Then we have $\LL(\ul-\vl)=f_{1,\lambda}$.

Notice the fact that, for any function $\phi$   and time interval
$I=[t_0,t_1]$
\[\|\phi_\lambda\|_{X^1[I]}\approx \lambda^s|I|^{-\frac{1}{2}}\|\phi_\lambda\|_{\txt[I]}+\lambda^s|I|^{\frac{1}{2}}\|\LL \phi_\lambda\|_{\txt[I]}.\]

So we get
\ben\|\ul\|_{\xl[J]}&\lesssim& \|\ul-\vl\|_{X^1[J]}+\|\vl\|_{S[J]}\\
&\lesssim& \lambda^{s}|J|^{-\frac{1}{2}}\|\ul-\vl\|_{\txt[J]}+\|f_{1,\lambda}\|_{\DI[J]}+\|f_{2,\lambda}\|_{DS[J]}\\
&\lesssim& \|\ul\|_{L^\infty_tH^s}
+\|f_{1,\lambda}\|_{\DI[J]}+\|f_{2,\lambda}\|_{DS[J]}. \een

Here we used the fact
\[\lambda^{s}|J|^{-\frac{1}{2}}\|\vl\|_{\txt[J]}\lesssim\|\vl\|_{S[J]},\]
which can be checked easily.

For the second estimate about local energy space, we can still
localize to fixed frequency, and need to show that
\beq\label{expressionforlocalenergy}\sup_j\sum^{|J|=\lambda^{4s+3}}_{
J\subset I}\|\chi^{\lambda}_j u_\lambda\|_{\xl[J]}^2\lesssim
\sup_j\sum^{|J|=\lambda^{4s+3}}_{ J\subset
I}(\lambda^{-2s-5}\|\chi_j^{\lambda}\partial_x
u_{\lambda}\|^2_{\txt[J]}+\|\chi^{\lambda}_j
f_\lambda\|_{\yl[J]}^2)\eeq To prove the estimate, let us consider
the inhomogeneous problem on interval $J=[a,b]$ of size
$|J|=\lambda^{4s+3}$,
\[ \LL \ul^k=P_\lambda \chi_k^\lambda f_\lambda, \hspace{0.2in} \ul^k(a)=\chi_k^\lambda u_{0,\lambda}\]
and prove that
\beq\label{provelocalenergyestimate}\|\chi^{\lambda}_j
\ul^k\|_{\xl[J]}\lesssim \left<j-k\right>^{-N} (
\lambda^s|J|^{-\frac{1}{2}}\|\chi_k^{\lambda} \ul^k\|_{L^2_{t,x}}
+\|\chi^{\lambda}_k f_\lambda\|_{\yl[J]}).\eeq When $j\approx k$,
it is essentially the same as (\ref{proveenergyestimate}). Notice
in the process of proving  (\ref{proveenergyestimate}), we get
\ben \|\ul\|_{\xl[J]}\lesssim
\lambda^s|J|^{-\frac{1}{2}}\|\ul\|_{L^2_{t,x}[J]}+\|f_\lambda\|_{\yl[J]}.\een
When $|j-k|\gg 1$, it follows from the rapid decay estimate on the
kernel $K_{jk} $ of $\chi_j^\lambda
e^{t\partial_x^3}P_\lambda\chi_k^\lambda$:
\[|K_{jk}(t,x,y)|\lesssim \lambda^{-N}\left<j-k\right>^{-N}, \hspace{0.2in }|t|\leq \lambda^{4s+3}.\]
Since $\ul=\sum_k \ul^k$, so we sum up $k$ in
(\ref{provelocalenergyestimate}), and get
\[ \|\chi^{\lambda}_j(x)\ul\|_{\xl[J]}\lesssim \sum_k\left<j-k\right>^{-N} ( \lambda^s|J|^{-\frac{1}{2}}\|\chi_k^{\lambda} \ul\|_{L^2_{t,x}[J]}
+\|\chi^{\lambda}_k(x)f_\lambda\|_{\yl[J]}),\] which is equivalent
to (\ref{expressionforlocalenergy}).
\end{proof}

\section{Linear and  bilinear estimate}
\label{sec3}

In this section, we look at solutions to the Airy equation,
\begin{equation}
\partial_t u+ \partial^3_xu=0, \hspace{3mm}  u(0,x)=u_0(x). \label{Airy}
\end{equation}
Solutions satisfy the following Strichartz and local smoothing
estimate \cite{KPV3}\cite{Tao}.

\begin{prop}Let $(q, r)$ be Strichartz pair
\begin{equation}\frac{2}{q}+\frac{1}{r}=\frac{1}{2}, \hspace{3mm} 4\leq q\leq\infty. \label{StrichartzPair}\end{equation}

 Then the solution of the Airy equation
satisfies
\[\|u\|_{L^q_tL^r_x}\lesssim \||D|^{-\frac{1}{q}}u_0\|_{L^2}.\]
\end{prop}

\begin{prop}
The solution of the Airy equation satisfies the local smoothing
estimate
\[\|\partial_x u\|_{L^{\infty}_xL^2_t}\lesssim \|u_0\|_{L^2}.\]
\end{prop}
\begin{prop}
The solution of the Airy equation satisfies maximal function
estimate
\[\|\partial^{-\frac{1}{4}}_x u\|_{L^{4}_xL^{\infty}_t}\lesssim \|u_0\|_{L^2}.\]
\end{prop}
Once we have estimates for linear equation, we can extend it
to $X^1$.
\begin{cor} Let $(q,r) $ be a
Strichartz pair as   in relation
\textnormal{(\ref{StrichartzPair})}.  Then we have \beq\|\eit
\ul\|_{L^q_tL^r_x}\lesssim
\lambda^{-\frac{1}{q}-s}\|\ul\|_{X^1[I]},\label{StrichartzX}\eeq
Also, the following smoothing estimate and maximal function
estimate hold \beq\|  \eit  \ul\|_{L^{\infty}_xL^2_t}\lesssim
\lambda^{-1-s}\|\ul\|_{X^1[I]}\label{localsmoothing},\eeq
\beq \|\eit
\ul\|_{L^{4}_xL^{\infty}_t}\lesssim\lambda^{\frac{1}{4}-s}\|\ul\|_{X^1[I]},\eeq
\end{cor}

\begin{proof}
The results follow by expanding
 $u_\lambda$ via Duhamel's formula.

If $\LL \ul =f$, then
\[\ul=e^{-t\partial^3_x}\ul(t_0)+\int_{t_0}^te^{-(t-s)\partial^3_x}f(s)ds.\]
From Strichartz estimate, and its dual form - the inhomogeneous
Strichartz estimate, see Theorem 2.3 in Tao~\cite{Tao} section
2.3,  and we get \ben
\|\eit \ul\|_{L^q_tL^r_x}&\lesssim& \|\eit e^{-t\partial^3_x}\ul(t_0) \|_{L^q_tL^r_x}+\lambda^{-\frac{1}{q}}\|\eit f\|_{L^1_tL^2_x}\\
&\lesssim&\lambda^{-\frac{1}{q}-s}\|\ul\|_{X^1[I]}. \een
We can prove the local smoothing and maximal estimate in the same
way. 
\end{proof}

We will also need the bilinear estimate as in  \cite{Axel}.
\begin{prop}\label{freebil}
Let $I^s_{\pm}$ be defined by its Fourier transform in the space
variable:
\[\mathcal{F}_xI^s_{\pm}(f,g)(\xi):=\int_{\xi_1+\xi_2=\xi}|\xi_1\pm\xi_2|^s \hat{f}(\xi_1)\hat{g}(\xi_2)d\xi_1.\]
Assume $u, v$ be two solutions to the Airy equation with initial
data $u_0,v_0$.  Then we have the bilinear estimate
\beq\|I^{\frac{1}{2}}_{+}I^{\frac{1}{2}}_{-}(u,v)\|_{L^2_{x,t}}\lesssim
\|u_0\|_{L^2_x}\|v_0\|_{L^2_x}.~\label{freebilinear}\eeq

\end{prop}
\begin{proof} For a solution to the Airy equation, we can write
down its Fourier transform,
\[
\widetilde{u}=\delta(\tau-\xi^3)\hat{u}_0,~~~~ \mbox{~~}
\widetilde{v}=\delta(\tau-\xi^3)\hat{v}_0.\]Then
\[\widetilde{I^{\frac{1}{2}}_{+}I^{\frac{1}{2}}_{-}(u,v)}(\tau,\xi)
=\int_{\xi_1+\xi_2=\xi}|\xi_1+\xi_2|^{\frac{1}{2}}|\xi_1-\xi_2|^{\frac{1}{2}}\hat{u}_0(\xi_1)\hat{v}_0(\xi_2)\delta(\tau-\xi_1^3-\xi_2^3)\,d\xi_1.\]
 Let us make change of variable
$\xi_1+\xi_2=\xi, \tau-\xi_1^3-\xi_2^3=\eta$.\\ With $\tau$, $\xi$
fixed, we have
\[d\xi_1=\frac{1}{3|\xi_1+\xi_2||\xi_1-\xi_2|}d\eta,\] hence we get
\[\widetilde{I^{\frac{1}{2}}_{+}I^{\frac{1}{2}}_{-}(u,v)}(\tau,\xi)=
\frac{1}{3|\xi_1+\xi_2|^{\frac{1}{2}}|\xi_1-\xi_2|^{\frac{1}{2}}}\hat{u}_0(\xi_1)\hat{v}_0(\xi_2).\]
Now $\xi_1$, $\xi_2$ are solutions to
\[ \xi_1+\xi_2=\xi, \hspace{0.2in}
\xi_1^3+\xi_2^3=\tau,\] So we have
\[d\tau d\xi=3|\xi_1^2-\xi_2^2|d\xi_1d\xi_2,\]
and it follows
\[\|I^{\frac{1}{2}}_{+}I^{\frac{1}{2}}_{-}(u,v)\|_{L^2_{x,t}}\lesssim \|u_0\|_{L^2_x}\|v_0\|_{L^2_x}.\]
\end{proof}
\begin{remark} Propostion~\ref{freebil}  gives us the usual $L^2$ estimate on product of two free solutions
 whenever they have frequency separation,
i.e. $|\xi_1\pm\xi_2|\neq 0$, $\xi_1\in supp{
\hspace{0.03in}\hat{u}}$, $\xi_2\in supp{
\hspace{0.03in}\hat{v}}$.

It is very useful especially when we localize the solutions into
dyadic frequency pieces, then the operators $I^{s}_{\pm}$ can be
simply replaced by scaler multiplication. We have the following
cases:

When $|\xi_1|\approx \mu,  |\xi_2|\approx \lambda, \mu\gg\lambda$,
then we get
\beq\|uv\|_{\txt}\lesssim\mu^{-1}\|u_0\|_{L^2}\|v_0\|_{L^2}.\label{bilinearhighlow}\eeq

When $|\xi_1|\approx |\xi_2|\approx \lambda$, and $\xi_1,\xi_2$
have opposite sign, so the output has frequency
$|\xi_1+\xi_2|\approx \alpha\lesssim \lambda$, then we get
\beq\|uv\|_{\txt}\lesssim\lambda^{-\frac{1}{2}}\alpha^{-\frac{1}{2}}\|u_0\|_{L^2}\|v_0\|_{L^2}.\label{bilinearhh}\eeq

\end{remark}

In case $|\xi_1|\approx |\xi_2|\approx \lambda$, but $\xi_1,\xi_2$
have same sign, the output lies close to a new curve
$\tau=\frac{1}{4}\xi^3$. Following the idea in~\cite{KPV1}, we
have the following Proposition.

\begin{prop} \label{propNC}Assume u,v are two smooth solutions to the Airy equation with initial data $u_0$, $v_0$,   localized at frequencies about
 the comparable size and also the same  sign,
and $I$ be an interval of size less than 1, then we have the
following estimate \beq \|\eit
uv\|_{X^{\frac{1}{4},\frac{1}{4}}_{\tau=\frac{1}{4}\xi^3}}
\lesssim \|u_0\|_{L^2_x}\|v_0\|_{L^2_x}.\label{bilinearNC}\eeq
\end{prop}
\begin{proof}
The proof is essentially the same as Proposition~\ref{freebil}.
There we first take two frequency really close, but have small
separation, i.e.  $|\xi_1-\xi_2|\geq\epsilon$, so that all the
calculation are still true, and we get the
estimate~(\ref{freebilinear}). Notice that
\[\xi_1+\xi_2=\xi, \hspace{0.2in} \xi_1^3+\xi_2^3=\tau.\]
So we solve for $\xi_1, ~~\xi_2$ and get
$|(\tau-\frac{1}{4}\xi^3)\xi|^{\frac{1}{2}}=\frac{3}{4}|(\xi_1+\xi_2)(\xi_1-\xi_2)|$,
which is exactly the multiplier we have in the space
$\dot{X}^{\frac{1}{4},\frac{1}{4}}_{\tau=\frac{1}{4}\xi^3}$. Then
we take the limit as $\epsilon\rightarrow 0$, and the norm
converges as long as we are considering smooth functions. So we
get \beq \|
uv\|_{\dot{X}^{\frac{1}{4},\frac{1}{4}}_{\tau=\frac{1}{4}\xi^3}}
\lesssim \|u_0\|_{L^2_x}\|v_0\|_{L^2_x}.\label{bilinearNC2}\eeq To
pass to
nonhomogeneous space
, notice the following estimate
\[\|\eit f\|_{\txt}\lesssim |I|^{\frac{1}{4}}\|f\|_{L^4_tL^2_x}\lesssim \|f\|_{\dot{X}^{0,\frac{1}{4}}_{\tau=\frac{1}{4}\xi^3}}.\]
The last inequality is by Sobolev embedding.
\end{proof}
In Proposition \ref{biliinearprop}, we will extend these estimates
(\ref{bilinearhighlow}) (\ref{bilinearhh}) from free solutions to
functions in $X^1$.


Now we list some $L^p$ estimates, which are mostly
straightforward.
\begin{prop}\label{easyestimate} When $-1\leq s\leq -\frac{3}{4}$, we have the following estimates.
\beq\|\eit Q_{\sigma}u_\lambda\|_{L^2_{x,t}}\lesssim
\sigma^{-1}\lambda^{-s}|I|^{-\frac{1}{2}}\|u_\lambda\|_{X^1},
\label{l2Asigma}\eeq
\begin{equation}\|Q_{\lambda^{4+\frac{3}{2s}}\lesssim \sigma\lesssim\lambda^3}u_\lambda\|_{L^2_{x,t}}\lesssim
\lambda^{-3s-\frac{3}{2s}-\frac{11}{2}}\|\ul\|_{\B} \lesssim
\lambda^{-2-s}\|u_\lambda\|_{\B},\label{L2B}\end{equation}
\begin{equation}\|Q_{\lambda^{4+\frac{3}{2s}}\lesssim \sigma\lesssim\lambda^3}u_\lambda\|_{L^\infty_{x,t}}
\lesssim
\lambda^{-2s-1}\|u_\lambda\|_{\B},\label{LinftyB}\end{equation}
\begin{equation}\|\ul\|_{L^3_{x,t}}\lesssim \lambda^{-\frac{1}{3}-2s-2}\|\ul\|_{\K},\label{L3K}\end{equation}
\beq\|Q_{\gtrsim\lambda^3}\ul\|_{L^2_{t,x}}\lesssim
\lambda^{-2s-3}\|\ul\|_{\C},\label{L2C} \eeq
\begin{equation}\|Q_{\gtrsim\lambda^3}\ul\|_{L^q_{t,x}}\lesssim \lambda^{1-\frac{4(s+2)}{q}}
\|\ul\|_{\C\cap\LP}, \mbox{~} 2\leq q\leq
p,\label{LqC}\end{equation}
\beq\|Q_{\gtrsim\lambda^3}\ul\|_{L^3_{t,x}}\lesssim
\lambda^{-\frac{4}{3}(s+1)-\frac{1}{3} }\|\ul\|_{Z},
\label{l3z}\eeq
\beq\|Q_{\gtrsim\lambda^3}\ul\|_{L^6_{t,x}}\lesssim
\lambda^{-\frac{2}{3}(s+1) +\frac{1}{3} }\|\ul\|_{Z}.
\label{l6z}\eeq

\end{prop}
\begin{proof}
The proofs are mostly simple.
(\ref{l2Asigma}) is by definition combined with the size of the
interval.
(\ref{L2B}) (\ref{LinftyB})
(\ref{L2C}) are consequences of Bernstein inequality. (\ref{LqC})
is  by interpolating the $L^2$ estimate with $L^p$.

The only nontrivial one is  (\ref{L3K}),
Similar to~\cite{Ta1}, we look at the operator $S(\sigma)$ defined
by multiplier
$e^{\sigma^2}\Gamma(\sigma)\frac{1}{(\tau-\frac{1}{4}\xi^3\pm
i0)^{\sigma}}$, where $\Gamma(\sigma)$ is the complex valued
Gamma-function. We claim that
\[S(0+iy)P_\lambda: \txt\rightarrow L^2_{x,t},\]
\[S(\frac{3}{2}+iy)P_\lambda:L^1_{x,t}\rightarrow \fxt.\]
Let us prove the second one by computing its   Fourier inversion,
\[\mathcal{F}^{-1}\frac{\theta_{\lambda}(\xi)}{(\tau-\frac{1}{4}\xi^3\pm i0)^{\frac{3}{2}+iy}}=
\mathcal{F}^{-1}{(\tau\pm i0)}^{-\frac{3}{2}-iy} \cdot
\mathcal{F}^{-1}[\theta_\lambda(\xi)\delta_{\tau=\frac{1}{4}\xi^3}].\]
$\theta_\lambda(\xi)$ is some smooth bump function around
$\xi=\lambda$, which we used to define $P_\lambda$.

From direct computation, we have
\[\|e^{(\frac{3}{2}+iy)^2}\Gamma(\frac{3}{2}+iy)\mathcal{F}^{-1}{(\tau\pm i0)}^{-\frac{3}{2}-iy}\|_{L^\infty}\lesssim t^{\frac{1}{2}},\]
and by stationary phase we get
\[\|\mathcal{F}^{-1}[\theta_\lambda(\xi)\delta_{\tau=\frac{1}{4}\xi^3}]\|_{L^\infty}=\|\int \theta_\lambda(\xi)e^{ix\xi+i\frac{1}{4}t\xi^3}\,d\xi\|_{L^\infty}\lesssim (t\lambda)^{-\frac{1}{2}}.\]
Combining them together,  we get
\[\|S(\frac{3}{2}+iy)P_\lambda\|_{L^1_{x,t}\rightarrow \fxt}
\lesssim \lambda^{-\frac{1}{2}}.\] Also notice the trivial bound
\[\|S(0+iy)P_\lambda\|_{\txt\rightarrow L^2_{x,t}}\lesssim C.\]
We interpolate to get
\[\|S(\frac{1}{2}+iy)P_\lambda\|_{L^{\frac{3}{2}}_{x,t}\rightarrow L^3_{x,t}}\lesssim \lambda^{-\frac{1}{6}}.\]
Define the operator $T$ by multiplier
$\frac{1}{(\tau-\frac{1}{4}\xi^3\pm i0)^{\frac{1}{4}}}$, and
$S(\frac{1}{2})=cTT^*$, $c=e^{\frac{1}{4}}\Gamma{(\frac{1}{2})}$.
So by the $TT^*$ argument~\cite{Tomas} ~\cite{Tao}, we have
$\|TP_\lambda\|_{\txt\rightarrow L^3_{x,t}}\lesssim
\lambda^{-\frac{1}{12}}$.
\\
Hence we get
\[\|\ul\|_{L^3_{x,t}}\lesssim \lambda^{-\frac{1}{3}-2s-2}\|\ul\|_{\K}.\]

If we take $q=3$ in (\ref{LqC}), combining with (\ref{L3K}) we get
(\ref{l3z}).

If we take $q=6$ in (\ref{LqC}), also compare with \ben
\|Q_{\sigma\approx\lambda^3}\ul\|_{L^6_{t,x}}\lesssim(\lambda\sigma)^{\frac{1}{6}}\|\ul\|_{L^3_{t,x}}\lesssim\lambda^{-2s-2+\frac{1}{3}}\|\ul\|_{\K},
\een we get (\ref{l6z}). From
Remark~\ref{functionspaceremark}(\ref{modulationmatch}), we only
put pieces in $\K$ norm when it lies close to the special curve,
and in that case its modulation is
 close to $\lambda^3$.
\end{proof}

Also, let us collect some bilinear estimates that will be very
useful in the next section.
\begin{prop}\label{biliinearprop}
For $\mu\gg\lambda\geq\alpha$, as before $\elt$ is the sharp
cutoff on time interval $I_\lambda$ of size $\Il=\lambda^{4s+3}$.
We have the following estimates:
 \beq\|\emt\um\vl\|_{\txt}\lesssim \mu^{-1-s}\lambda^{-s}\|\um\|_{\xmm}\|\vl\|_{\xll},\label{L2XX}\eeq
 \beq \|\emt P_{\approx \lambda}(\um\vm)\|_{\txt}\lesssim \mu^{-\frac{1}{2}-2s}\lambda^{-\frac{1}{2}}\|\um\|_{\xmm}\|\vl\|_{\xmm}
,\label{L2XXHHL}\eeq
\beq\|\elt \ul\va\|_{\txt}\lsm
\max{\{\lambda^{-\frac{1}{3}-2s-2}\alpha^{-\frac{1}{6}-s},\lambda^{-2-s}\alpha^{\frac{1}{2}-s}\}}\|\ul\|_{S[\ill]}\|\va\|_{X^1[\iaa]}.\label{L2SSX}\eeq

\end{prop}
\begin{proof} For
(\ref{L2XX})  and (\ref{L2XXHHL}), we expand $u, v$ via Duhamel's
formula, and apply the bilinear estimates (\ref{bilinearhighlow})
(\ref{bilinearhh}) repeatedly. See~\cite{CKSTT2} Lemma 3.4 for a
similar proof.

For (\ref{L2SSX}), we still break $\ul$ by the size of modulation,
and see that the worst estimate comes when $\ul\in\K\cap\LP$. Then
we use $L^3$ for $\ul$, and $L^6$ for $\va$. \ben\|\el
\ul\va\|_{\txt}\lesssim\|\el\ul\|_{\txt}\|\ea\va\|_{\fxt}\lesssim
\lambda^{-2s-3}\alpha^{\frac{1}{2}-s}\|\ul\|_{\C[\ill]}\|\va\|_{\xaa},\een
\ben\|\el
\ul\va\|_{\txt}\lesssim\|\el\ul\|_{L^3_{t,x}}\|\ea\va\|_{L^6_{t,x}}\lesssim
\lambda^{-\frac{1}{3}-2s-2}\alpha^{-\frac{1}{6}-s}\|\ul\|_{\K[\ill]}\|\va\|_{\xaa}.\een
By comparing the coefficients in the   estimates above, we get
\beq\|\elt \ul\va\|_{\txt}\lsm
\lambda^{-\frac{1}{3}-2s-2}\alpha^{-\frac{1}{6}-s}\|\ul\|_{Z[\ill]}\|\va\|_{X^1[\iaa]}.\label{L2SX}\eeq

If we also consider the case of $\ul\in \B$,
\beq\label{L2BXHL}\|\el
\ul\va\|_{\txt}\lesssim\|\el\ul\|_{L^2_{t,x}}\|\ea\va\|_{L^\infty_{t,x}}\lesssim
\lambda^{-2-s}\alpha^{\frac{1}{2}-s}
\|\ul\|_{\B[\ill]}\|\va\|_{\xaa}\eeq and compare the coefficients,
we get (\ref{L2SSX}).
\end{proof}
\begin{remark}\label{remarkonSS}
We don't have a good $L^2$ estimate on the product of two pieces
both in  $S$. But we will still list here some of the cases, which
are manageable.


When $\ul, \va\in \B$, bound $\ul$ in $\txt$, and $\va$ in $\fxt$.
\beq\label{L2BB}\|\el
\ul\va\|_{\txt}\lsm\lambda^{-2-s}\alpha^{-2s-1}\|\ul\|_{\B[\ill]}\|\va\|_{\B[\iaa]}.\eeq

When $\ul, \va\in Z$, bound $\ul$ in $L^3$, and $\va$ in $L^6$, we
get \beq\|\elt \ul\va\|_{\txt}\lsm
\lambda^{\threez}\alpha^{\sixz}\|\ul\|_{\zll
}\|\va\|_{\zaa}.\label{L2SS}\eeq

When $\ul\in Z, \va\in\B$, bound $\ul$ in $L^3$, and $\va$ in
$L^6$ which comes from Bernstein together with $L^2$ bound, we get
\beq\label{L2ZHBL}\|\el \ul\va\|_{\txt}\lsm\lambda^{\threez}
\max{\{\alpha^{-s-1},\alpha^{-\frac{5}{3}-2s}\}}\|\ul\|_{Z[\ill]}\|\va\|_{\B[\iaa]}.\eeq
The above three inequalities imply that \beq\|\elt
\ul\va\|_{\txt}\lsm
\lambda^{\threez}\alpha^{\sixz}\|\ul\|_{S[\ill]
}\|\va\|_{S[\iaa]}.\label{L2SSS}\eeq is true except for the case
$\ul\in \B, \va\in Z$, which corresponds to case the
high-frequency low modulation interacting with low-frequency high
modulation.

To estimate the case  $\ul\in \B, \va\in Z$, use $L^2$ on $\ul$,
$L^\infty$ on $\va$, and we get \beq\label{l2ulbvaz} \|\el
\ul\va\|_{\txt}\lsm \|
\el\ul\|_{\txt}\|\el\va\|_{L^\infty_{t,x}}\eeq \ben
\hspace{1.6in}\lsm\lambda^{-2-s}\alpha\|\ul\|_{\B[\ill]}\|\va\|_{Z[\iaa]}.\een
The bound here is worse than the one in (\ref{L2SSS}).
\end{remark}

\section{Estimating the nonlinearity }
\label{sec4}

The goal of this part is to estimate the nonlinearity as in
Proposition~\ref{bilinearprop}. Since functions in $\xx$ have
different piece measured differently, we show that the estimate
\beq\|\partial_x(uv)\|_{Y^s\cap Y^s_{le}}\lesssim \|u\|_{X^s\cap
X^s_{le}}\|v\|_{X^s\cap X^s_{le}}\label{primaryestimate}\eeq is
almost true except for one special case.

Let us expand the estimate (\ref{primaryestimate}), the energy
norm takes the form \ben &&\sum_{\lambda\geq
1}\sup_{|J|=\lambda^{4s+3}, J\subset
[0,1]} \| \eta_{J}(t)P_\lambda(\partial_x(uv))\|_{Y_{\lambda}[J]}^2\\
&\lesssim& \sum_{\lambda\geq 1}\sup_{|J|=\lambda^{4s+3},J\subset
[0,1] } \| \eta_{J}(t)\lambda \sum_{\alpha\ll\lambda} P_\lambda(\ul\va)\|_{Y_{\lambda}[J]}^2\\
&+& \sum_{\lambda\geq 1}\sup_{|J|=\lambda^{4s+3},J\subset [0,1]}
\| \eta_{J}(t)\lambda
\sum_{\mu\gtrsim\lambda}P_\lambda(\um\vm)\|_{Y_{\lambda}[J]}^2.
\een We can do same expansion for the local energy norm.

In the case of high-low frequency interaction, our goal would be
to prove \beq\|\lambda   \ej(t) P_\lambda(\ul\va)\|_{\yl[J
]}\lesssim  C \|  \ul\|_{\xl[J]} \|
\va\|_{X_\alpha[K]}.\label{provehl}\eeq Here
$C=C(\lambda,\alpha)\lesssim 1$, and $K$ is a time interval with
size $ \alpha^{4s+3}$, so that $J\subset K $.

Now given (\ref{provehl}), we get bound for energy norm in the
case of high-low interaction \ben &&\sum_{\lambda\geq
1}\sup_{|J|=\lambda^{4s+3},J\subset [0,1]} \|
\eta_{J}(t)\lambda\sum_{\alpha\ll\lambda}P_\lambda(\ul\va)\|_{Y_{\lambda}[J]}^2
\\
&\lesssim& \sum_{\lambda\geq
1}C(\lambda,\alpha)^2\sup_{|J|=\lambda^{4s+3},J\subset [0,1]}  \|
\ul\|^2_{\xl[J]}\sum_{\alpha\ll\lambda}\|\va\|_{X^s}^2
\\
&\lesssim& \|u\|_{X^s}^2\|v\|_{X^s}^2. \een And we can prove  a
spatial localized version of (\ref{provehl}) in exactly the same
way. \beq \|\lambda   \chi^\lambda_j(x)\ej(t)
P_\lambda(\ul\va)\|_{\yl[J ]}\lesssim  C \| \chi^\lambda_j(x)
\ul\|_{\xl[J]} \|   \va\|_{X_\alpha[K]}.\label{provehlspatial}
\eeq Then we also get bound for local energy norm in the case of
high-low interaction \ben &&\sum_{\lambda\geq
1}\sup_j\sum_{|J|=\lambda^{4s+3},J\subset
[0,1]}\|\chi^{\lambda}_j(x)\eta_{J}(t)\lambda\sum_{\alpha\ll\lambda}P_\lambda(\ul\va)\|_{Y_{\lambda}[J]}^2
\\
&\lesssim& \sum_{\lambda\geq
1}C(\lambda,\alpha)^2\sup_j\sum_{|J|=\lambda^{4s+3},J\subset
[0,1]} \|\chi^{\lambda}_j(x)
\ul\|^2_{\xl[J]}\sum_{\alpha\ll\lambda}\|\va\|_{X^s}^2
\\
&\lesssim& \|u\|_{X^s_{le}}^2\|v\|_{X^s}^2. \een
%
%
%
One remark here is that we secretly turn the summation of $\alpha$
from $l^1$ to $l^2$ summation, which is not true in general.
Luckily, in our proof for (\ref{provehl}), the bound
$C(\lambda,\alpha)$ mostly involves negative power of $\alpha$ or
$\lambda$, which makes the summation valid. The only case worth
attention is case 1.1(b), where we illuminate  the $\alpha$
summation in detail.

In the case of high-high frequency interaction, we need to measure
each $\um$ on smaller time interval $\imm\subset J$, of size
$|\imm |=\mu^{4s+3}$.

We will prove the estimate \beq\| \lambda \ej
P_\lambda(\um\vm)\|_{\yl[J ]}\lesssim C
\hspace{.1in}\sup_{\imm\subset J} \| \um\|_{X_\mu[\imm]} \|
\vm\|_{\xx},\label{provehh}\eeq and its corresponding spatial
localized version \beq\| \lambda \ej\chi_j^\lambda(x)
P_\lambda(\um\vm)\|_{\yl[J ]}\lesssim C
\hspace{.1in}\sup_{\imm\subset J} \|\chi^{\mu}_{k(j)}
\um\|_{X_\mu[\imm]} \| \vm\|_{\xx}.\label{provehhspatial}\eeq Here
$\chi_{k(j)}^\mu(x)$ is a chosen spatial cutoff so that
$\chi_{j}^\lambda(x)\leq \chi_{k(j)}^\mu(x)$
  (we might need  two adjacent spatial cutoffs), $C=C(\lambda,\mu)\lesssim 1$.

Given (\ref{provehh}), we get the bound for energy norm in the
case of high-high interaction \ben &&\sum_{\lambda\geq
1}\sup^{|J|=\lambda^{4s+3}}_{J\subset [0,1]} \| \lambda
\eta_{J}(t)
\sum_{\mu\gtrsim\lambda}P_\lambda(\um\vm)\|_{Y_{\lambda}[J]}^2
\\
&\lesssim& \sum_{\lambda\geq 1}
C(\lambda,\mu)^2\left(\sum_{\mu\gtrsim\lambda}\sup^{|\imm|=\mu^{4s+3}}_{\imm\subset
[0,1]} \|  \um\|^2_{X_\mu[\imm]}\right)\sum_{\mu\gtrsim\lambda}\| \vm\|^2_{\xx}\\
&\lesssim& \| \um\|^2_{X^s}\| \vm\|^2_{\xx}. \een And with
(\ref{provehhspatial}), we can bound the local energy norm in the
case of high-high interaction \ben &&\sum_{\lambda\geq
1}\sup_j\sum_{|J|=\lambda^{4s+3},J\subset
[0,1]}\|\chi^{\lambda}_j(x)\eta_{J}(t)\lambda
\sum_{\mu\gtrsim\lambda}P_{\lambda}(\um\vm)\|_{Y_{\lambda}[J]}^2
\\
&\lesssim& \sum_{\lambda\geq
1}C(\lambda,\mu)^2\left(\sum_{\mu\gtrsim\lambda}\sup_{k(j)}\sum_{|\imm|=\mu^{4s+3},
\imm\subset [0,1]}\|\chi^{\mu}_{k(j)}(x)
  \um\|^2_{X_\mu[\imm]}\right)\sum_{\mu\gtrsim\lambda}\| \vm\|^2_{\xx}\\
&\lesssim& \| \um\|^2_{X^s_{le}}\| \vm\|^2_{\xx}. \een In both of
the estimates, we need change the order of $\lambda$, $\mu$
summation. Luckily the bound $C(\lambda,\mu)$ in (\ref{provehh})
(\ref{provehhspatial}) will help us to perform the $\lambda$
summation.


Since   the proofs for (\ref{provehlspatial})
(\ref{provehhspatial})
 are essentially the same as (\ref{provehl}) (\ref{provehh}). We will discard the spatial cutoff in our proofs unless needed.
\begin{remark} To be more precise, for  spatial localization, instead of  writing a function as $\ul=\sum_j\chi_j^\lambda(x)\ul$,
we need to decompose each function as
\[\ul=\sum_j u_{\lambda,j}, \hspace{.5in} u_{\lambda,j}=P_\lambda(\chi_j^\lambda \ul).\]
In this way, we preserve the frequency localization while blurring
the spatial localization. But thanks  to the fast decay property
of the kernel of $\chi^\lambda_k(x)P_\lambda \chi^\lambda_j(x)$,
we have
\[|\chi_k^\lambda u_{\lambda,j}|\lesssim |k-j|^{-N}\lambda^{-N}\|\chi_j^\lambda\ul\|_{L^\infty_tL^2_x}, \hspace{0.3in}{|k-j|\gg 1}.\]
So the difference of the two decompositions is really negligible.
Similar reasoning applies when we interchange the modulation
localization and time localization.
\end{remark}
%

Before getting into detail, notice that
$\widetilde{uv}(\tau,\xi)=\widetilde{u}(\tau_1,\xi_1)\ast\widetilde{v}(\tau_2,\xi_2)$,
so we have \[\tau=\tau_1+\tau_2,\hspace{0.2in}\xi=\xi_1+\xi_2,\]
and the resonance identity
\begin{equation}
\tau-\xi^3=(\tau_1-\xi_1^3)+(\tau_2-\xi_2^3)-3\xi\xi_1\xi_2.
\end{equation}
Also, the following high modulation relation is quite useful in
our proof.
\begin{equation}
\sm=\max(|\tau-\xi^3|, |\tau_1-\xi_1^3|, |\tau_2-\xi_2^3|)\gtrsim
|\xi\xi_1\xi_2|. \label{HiMode}
\end{equation}
This relation forces high modulation either on the input or on
output, which gives a gain.
\subsection{Estimate for $X^1\times X^1$.} When $u, v\in X^1$, we
break them into dyadic pieces and discuss the problem in different
cases. As pointed out in Remark
~\ref{functionspaceremark}(\ref{remarkonx}), for function $\ul\in
X^1[I_\lambda]$, $|I_\lambda|=\lambda^{4s+3}$, we think of it as
its extension $u_{\lambda,E}$, which is defined on the whole real
time line and still  supported on neighborhood of $I_\lambda$.

\noindent\textbf{Case 1.1:} High-Low frequency interaction.
Suppose $\lambda \gg \alpha$, then the output frequency is
$\lambda$. From~(\ref{HiMode}), let $M=\lambda^2\alpha$, then
\[\lambda \el \ua\vl=\sum_{Q_i\in \{Q_{\gtrsim M}, Q_{\ll
M}\}}\lambda\el Q_1[Q_2 \ua Q_3 \vl].\] Clearly in each term, at
least one of $Q_i$ must be $\QM$.

\noindent\textbf{Case 1.1(a):} When high modulation comes from
input,
 simply   bound that piece in $L^2$  and the other in $L^\infty$. Combining with Bernstein inequality, we get
\ben \|\lambda \el Q_1[Q_{2} \ua Q_3
\vl]\|_{Y_\lambda[\ill]}&\lesssim&
\|\lambda\el Q_1[Q_{2}\ua Q_{3}\vl]\|_{\DI[\ill]}\\
&\lesssim
&\lambda^{}\alpha^{-s}M^{-1}(\alpha^{\frac{1}{2}}+\lambda^{\frac{1}{2}})
\|\ua\|_{X^1[\iaa] }\|\vl\|_{X^1[\ill] }\\
&\lesssim& \lambda^{-\frac{1}{2}}\alpha^{-1-s}\|\ua\|_{X^1[\iaa]
}\|\vl\|_{X^1[\ill] }. \een For $s\geq -1$ we can sum up with
respect to $\alpha$ then $\lambda$.

\noindent\textbf{Case 1.1(b):} If none of $Q_2, Q_3$ have high
modulation, this forces $Q_1=Q_{\approx M}$. Depends on the size
of $M$, we bound the output in different spaces ($\DI$ or
$X^{-s,s}$).
Using the bilinear estimate (\ref{L2XX}), we have \ben\|\lambda
\el Q_{\approx M\lesssim \lambda^{4+\frac{3}{2s}}}[Q_{2}\ua Q_3
\ul]\|_{Y_\lambda[\ill]}
&\lesssim&\lambda^{1+s}\Il^{\frac{1}{2}}\lambda^{-1-s}\alpha^{-s}\|\ua\|_{X^1[I_\alpha]}\|\vl\|_{X^1[I_\lambda]}
\\
&\lesssim&
\|\ua\|_{X^1[I_\alpha]}\|\vl\|_{X^1[I_\lambda]}\label{l1l2},\een
\ben \|\lambda \el Q_{\approx M\geq
\lambda^{4+\frac{3}{2s}}}[Q_{2}\ua Q_3 \ul]\|_{Y_\lambda[\ill]}
&\lesssim&\lambda^{1-s}M^{s}\lambda^{-1-s}\alpha^{-s}\|\ua\|_{X^1[I_\alpha]}\|\vl\|_{X^1[I_\lambda]}
\\
&\lesssim&
\|\ua\|_{X^1[I_\alpha]}\|\vl\|_{X^1[I_\lambda]}.\label{carefulsummation}
\een
\begin{remark} We need to be careful with $\alpha$ summation in above estimates. For the first one we use factor $\alpha^s$ to turn $l^1$ summation to $l^2.$
A careful way of doing the second one is to write the modulation
as a multiple of $\lambda^2\alpha$, and use the $l^2 $ summability
of modulation.
\begin{eqnarray*}\|\sum_{\alpha\ll\lambda}\sum_{\theta}\lambda\el Q_{
(\lambda^2\alpha)\theta}(u_\lambda v_\alpha)\|^2_{\yl} &\lsm&
(\sum_{\theta} \|\sum_{\alpha\ll\lambda}\lambda \el Q_{
(\lambda^2\alpha)\theta}(u_\lambda v_\alpha)\|_{DS})^2\\
&\lsm& \{  \sum_{\theta} (\sum_{\alpha\ll\lambda}\|\lambda\el Q_{
(\lambda^2\alpha)\theta}(u_\lambda v_\alpha)\|_{DS}^2)^{1/2}\}^2
\\ &\lesssim&
\{\sum_{\theta}(\sum_{\alpha\ll\lambda}\theta^{2s}\|
u_\lambda\|^2_{X^1[\ill]}\| v_\alpha\|^2_{X^1[\iaa]})^{1/2}\}^2
\\
&\lesssim&
(\sum_{\theta}\theta^{s})^{\frac{1}{2}}\|u_\lambda\|_{X^1[I_\lambda]}^2\|v
\|^2_{X^s}.
\end{eqnarray*}
In second inequality, since the modulation is different, we do
have the $l^2$ summation.
\end{remark}

\noindent\textbf{Case 1.2:} High-High frequency interaction with
low frequency output, $\lambda\ll \mu$. Here we need to cut the
interval $I_\lambda$ into finer scale so that $\um$ is measured on
smaller intervals $\imm$.
%
%
\[\um=\sum_i \um^i,\hspace{0.2in}\um^i\in X^1[I_\mu^i],\hspace{0.2in}\cup I^i_{\mu}=I_\lambda.\]
Then the output has the expression
\[\lambda\el \um\vm=\sum_{i}\sum_{Q_j\in \{Q_{\gtrsim \lambda\mu^2}, Q_{\ll \lambda\mu^2}\}}\lambda Q_1[Q_2 \um^i Q_3 \vm^i ].\]
\noindent\textbf{Case 1.2(a):} When $Q_1=Q_{\gtrsim
\lambda\mu^2}$, we place the output in   $DZ[I_\lambda]$, by using
(\ref{L2XXHHL}) \ben
 \|   Q_2 \um^i  Q_3 \vm^i\|_{\txt}
\lesssim \lambda^{-\frac{1}{2}}\mu^{-\frac{1}{2}}\mu^{-2s}
\|\um\|_{X^1[I_\mu^i]}\|\vm\|_{X^1[I_\mu^i]} \een and the almost
orthogonality of  the product $\lambda Q_{\sigma}(\um^i\vm^i)$
with $\lambda Q_{\sigma}(\um^j\vm^j)$, we get \ben
\|\sum_{i}\lambda Q_{\sigma\gtrsim \lambda\mu^2}[Q_2 \um^i Q_3
\vm^i]\|_{\DC[I_\lambda]}
 &\lesssim&\lambda^{2-4s}(\sigma_{\gtrsim
\lambda\mu^2})^{2s+1}|\frac{\ill}{\imm}|^{\frac{1}{2}}\|\um^i
\vm^i \|_{\txt}
\\
&\lesssim&
\sup_{\imm^i\subset \ill}\|\um\|^2_{X^1[I_\mu^i]}\|\vm\|^2_{X^s}.
\een The $DZ$ norm also has the $L^p$ component.  Here because the
modulation is high, we can interchange interval and modulation
cutoff
 and have $l^p$ summation of the intervals. Using Strichartz estimates (\ref{StrichartzX})
  and Bernstein inequality on the product, we get
\ben &&\|\sum_{i}\lambda Q_{\sigma\gtrsim \lambda\mu^2}[Q_2 \um^i
Q_3 \vm^i]\|_{\DLP[I_\lambda]}
\\ &\lesssim&\sup_{\imm^i}\|\lambda
 Q_{\sigma\gtrsim \lambda\mu^2}[Q_2\um^i Q_3  \vm^i]\|_{\DLP[I^i_\mu]}
\\
&\lesssim& \sup_{\imm^i}\frac{\lambda}{\lambda\mu^2}
 \|  \um\|_{L^\infty_tL^2_x[\imm^i]}\|\vm\|_{L^\infty_tL^2_x[\imm^i]}\\
&\lesssim&\mu^{{-2(s+1)} }\sup_{\imm\subset
\ill}\|\um\|_{X^1[I_\mu]}\|\vm\|_{X^s}. \een Because of the
summation on $\lambda$ here, we have only $s>-1$ in
Proposition~\ref{bilinearprop}, but not at the endpoint $s=-1$.

\noindent\textbf{Case 1.2(b):} When input has high modulation,  we
use the local energy space to get good control of the interval
summation.

Before that, let us state a useful lemma:
\begin{lemma}\label{combinecases} Suppose $-1\leq s\leq -\frac{3}{4}$, \hspace{0.1in}$0\leq k\leq \frac{1}{2}$, \hspace{0.05in}then we have
\[\|  Q_{\sigma\gtrsim\lambda^3}f_\lambda\|_{DZ[\ill] }\lesssim \sup_{\sigma\gtrsim\lambda^3}\lambda^{2s+3k}\sigma^{-k}\|Q_{\sigma}f_\lambda\|_{\txt[\ill]},\]
\[\| f_\lambda\|_{\yl[\ill] }\lesssim \sup_{\sigma}\lambda^{3s+\frac{3}{2}+3k}\sigma^{-k}\|Q_{\sigma}f_\lambda\|_{\txt[\ill]}.\]
\end{lemma}
\begin{proof} From the definition of $\yl[\ill]$, we just need to bound different modulation in suitable spaces, and compare the bounds with the ones in our lemma.
The $DZ$ norm also has  $L^p$ component, we use Bernstein to turn
$L^p$ into $L^2$ norm.
%
\end{proof}
\begin{remark}
These estimates are very crude. When applying on the nonlinearity,
we might need to do modulation analysis, or use better interval
summation in some cases, e.g. case 1.2(a). But when one of the
inputs has high modulation, a simple $L^2$ estimate saves us from
tedious case by case calculation.
\end{remark}
Let us first bound the spatial localized output in $L^2$.
%
\ben & & \|\lambda\el\cljx
Q_\sigma[\sum_{i}(Q_{\gtrsim\lambda\mu^2}\um^i)
(Q_3\vm^i)]\|^2_{\txt[I_\lambda]}
\\
&\lesssim&  \sigma \|\lambda \el\cljx [\sum_{i}(Q_{\gtrsim\lambda\mu^2}\um^i) ( Q_3\vm^i)]\|^2_{L^2_xL^1_t}\\
&\lesssim& \lambda^2\sigma\sum_i\|\cljx
Q_{\gtrsim\lambda\mu^2}\um^i\|^2_{\txt[I_\mu^i]}
\sum_i\|\cljx Q_3 \vm^i\|^2_{\fxtt[I_\mu^i]}\\
&\lesssim& \lambda^2\sigma |\frac{I_\lambda}{I_\mu}|\sup_i \|\cljx
Q_{\gtrsim\lambda\mu^2} \um^i
\|^2_{\txt[I^i_\mu]}\sup_j\sum_i\|\cljx Q_3 \vm^i\|^2_{\fxtt[I^i_\mu]}\\
&\lesssim&\sigma\lambda^{4s+3}\mu^{-12s-12}\sup_{\imm}\|\chi^\mu_{k(j)}(x)\um\|^2_{X^1[I_\mu]}\|\vm\|^2_{X^s_{le}}.
\een To get same estimate without the spatial localization, we
need to sum up $j$ \ben & & \sum_j\|\lambda\el\cljx
Q_\sigma[\sum_{i}(Q_{\gtrsim\lambda\mu^2}\um^i)
(Q_3\vm^i)]\|^2_{\txt[I_\lambda]}
\\
&\lesssim&\lambda^2\sigma\sum_{i,j}\|\cljx Q_{\gtrsim\lambda\mu^2}\um^i\|^2_{\txt[I^i_\mu]}\sup_j\sum_i\|\cljx Q_3 \vm^i\|^2_{\fxtt[I^i_\mu]}\\
&\lesssim& \lambda^2\sigma |\frac{I_\lambda}{I_\mu}|\sup_i  \|
Q_{\gtrsim\lambda\mu^2} \um^i
\|^2_{\txt[I^i_\mu]}\sup_j\sum_i\|\cljx Q_3 \vm^i\|^2_{\fxtt[I^i_\mu]}\\
&\lesssim&\sigma\lambda^{4s+3}\mu^{-12s-12}\sup_{\imm}\|
\um\|^2_{X^1[I_\mu]}\|\vm\|^2_{X^s_{le}}. \een By
Lemma~\ref{combinecases} $k=\frac{1}{2}$, we have the following
estimate with or without spatial localization. \ben \|\lambda \el
[\sum_{i}Q_{\gtrsim\lambda\mu^2}\um^i Q_3\vm^i]\|_{\yl[\ill]}
&\lesssim& \lambda^{5s+9/2}\mu^{-6s-6}\|
u_{\mu}\|_{X^1[I_\mu]}\|v_{\mu}\|_{X^s_{le}}. \een

We can sum up frequency $\lambda$ and $\mu$ when $-1\leq s\leq -\frac{3}{4}$.\\
\begin{remark} The estimates above demonstrate how we can use
local energy norm to get good interval summations, especially in
the case when time truncation blur the output modulation too much.
\end{remark}

\noindent\textbf{Case 1.3:} High-High frequency interaction giving
out the output of the same size. Now high modulation
(\ref{HiMode}) means $\lambda^3$.
\[\lambda \el \ul\vl=\sum_{Q_i\in
\{Q_{\gtrsim \lambda^3}, Q_{\ll \lambda^3}\}}\lambda\el Q_1[Q_2\ul
Q_3\vl].\] \noindent\textbf{Case 1.3(a):} When high modulation
comes from input, we estimate the output in $\DI$ \ben
\|\lambda \el Q_1[Q_{\sigma\gtrsim{\lambda^3}}\ul Q_3 \vl]\|_{\DI}&\lesssim & \lambda^{1+s}{\Il}^{\frac{1}{2}}\|Q_{\sigma\gtrsim{\lambda^3}}(\el\ul)\|_{\txt}\|\el\vl\|_{\fxt}\\
&\lesssim&\lambda^{-\frac{3}{2}-s}\|\ul\|_{X^1[\ill]}\|\vl\|_{X^1[\ill]}.
\een \noindent\textbf{Case 1.3(b):} When inputs have low
modulation, this forces the output to have modulation
approximately $\lambda^3$. In fact,  the output has  Fourier
support lying closer to another curve $\tau=\frac{1}{4}\xi^3$. To
give a good bound in this case, we want to prove \beq \|\lambda
P_\lambda(\ul\vl)\|_{\DK}\lesssim
\|\ul\|_{X^1[\ill]}\|\vl\|_{X^1[\ill]}.\label{goodNC} \eeq To do
this, let us  use another space~\cite{KT1}
\[\|u\|_{\dot{X}^{s,\frac{1}{2},1}}=\sum_{\vartheta} \left(\int_{|\tau-\xi^3|=\vartheta}|\tilde{u}(\tau,\xi)|^2\xi^{2s}|\tau-\xi^3|\,d\xi \,d\tau\right)^{\frac{1}{2}},\]
and claim the embedding inequality
\beq\|\ul\|_{\dot{X}^{s,\frac{1}{2},1}}\lesssim
\|\ul\|_{X^1[\ill]},\label{xembedding}\eeq which is proved by
looking at the extension $u_{\lambda,E}$,  and   definitions of
both norms.

Now for functions in $\dot{X}^{s,\frac{1}{2},1}$, we use
foliation. The idea is same as in Chapter 2.6 Lemma 2.9 in
Tao~\cite{Tao}. From Fourier inversion, we have
\[\ul(t,x)=\frac{1}{(2\pi)^2}\int \int \tilde{u}_{\lambda}(\tau,\xi)e^{it\tau+ix\xi}\,d\tau\,d\xi.\]
Then if we write $\tau_0=\tau-\xi^3$, we will have the foliation
\[\ul(t,x)=\frac{1}{2\pi}\int e^{it\tau_0}e^{t\partial_x^3}f_{\tau_0}\,d\tau_0,\]
where \[e^{t\partial_x^3}f_{\tau_0}=\frac{1}{2\pi}\int
\tilde{u}_\lambda(\tau_0+\xi^3,\xi)e^{it\xi^3+ix\xi}\,d\xi,\] and
$f_{\tau_0}$ has frequency about size $\lambda$, modulation about
size $\tau_0$.

Similarly we write down $\vl=\ul(t,x)=\frac{1}{2\pi}\int
e^{it\tau_0}e^{t\partial_x^3}g_{\tau_0'}\,d\tau_0'$.

Now using (\ref{bilinearNC2}) and Minkowski inequality \ben
\|\lambda
P_\lambda(\ul\vl)\|_{|\LL|^{-1}|D|^{-2s-2}\dot{X}^{\frac{1}{4},\frac{1}{4}}_{\tau=\frac{1}{4}\xi^3}}
&\lesssim& \frac{\lambda^{2s+3}}{\lambda^3}\sum_{\tau_0,\tau_0'}\int\!\!\int\|e^{t\partial_x^3}f_{\tau_0}e^{t\partial_x^3}g_{\tau_0'}\|_{\dot{X}^{\frac{1}{4},\frac{1}{4}}_{\tau=\frac{1}{4}\xi^3}}   \,d\tau_0\,d\tau_0'\\
&\lesssim&\lambda^{2s}\sum_{\tau_0,\tau_0'}\int\!\!\int \|f_{\tau_0}\|_{L^2_x}\|g_{\tau_0'}\|_{L^2_x}\,d\tau_0\,d\tau_0'\\
&\lesssim&
\|\ul\|_{\dot{X}^{s,\frac{1}{2},1}}\|\vl\|_{\dot{X}^{s,\frac{1}{2},1}}.
\een With the time cutoff we can pass to nonhomogeneous space, as
in Proposition \ref{propNC}. Combining with the embedding
(\ref{xembedding}), we proved (\ref{goodNC}).

\subsection{Estimate for $S\times S$.} When $u,v\in S$, we still need to consider different frequency interaction.
Notice that because of
Remark~\ref{functionspaceremark}(\ref{modulationmatch}), we only
consider pieces that have relatively high modulation:
$|\tau-\xi^3|\gtrsim|\xi|^{4+\frac{3}{2s}}$

\noindent\textbf{Case 2.1:}  High low frequency interaction.  The
nonlinearity look like $\lambda \el \ul\va$,  $\lambda\gg\alpha$.
As discussed in Remark~\ref{remarkonSS}, we don't have a good
bilinear estimate, but ($\ref{L2SSS}$) breaks down only for one
case.

\noindent\textbf{Case 2.1.1:} If $\ul, \va\in\B$, or $\ul\in Z,
\va\in \B$, or $\ul,\va\in Z$, we  can still use the $L^2$
estimate ($\ref{L2SSS}$) and Lemma~\ref{combinecases} with $k=0$
to get
\ben\|\lambda \el\ul\va\|_{\yl[\ill]}&\lesssim& \lambda^{3s+\frac{3}{2}}  \lambda^{1\threez}\alpha^{\sixz}\|\ul\|_{\zll}\|\va\|_{\zaa}\\
&\lesssim&\lambda^{\frac{1}{6}+s+\frac{2}{3}(s+1)}\alpha^{\sixz}\|\ul\|_{\zll}\|\va\|_{\zaa}.
\een Notice that the exponents add up to $-\frac{3}{2}-s<0$, we
can still sum up frequencies.

\noindent\textbf{Case 2.1.2:} Now if $\ul\in \B, \va\in Z$, where
(\ref{L2SSS}) failed. We use $L^2$ on $\ul$ and $L^3_tL^\infty_x$
on $\va$, still by Bernstein, \ben \|\lambda
Q_\sigma(\el\ul\va)\|_{\txt}\lesssim
\lambda\alpha^{\frac{1}{3}}\sigma^{\frac{1}{3}}\lambda^{-2-s}\alpha^{\threez}\|\ul\|_{\B[\ill]}\|\va\|_{Z[\iaa]},
\een so we from Lemma~\ref{combinecases}, we get \ben \|\lambda
\el\ul\va\|_{\yl[\ill]}\lesssim
\lambda^{2s+\frac{3}{2}}\alpha^{-\frac{4}{3}(s+1) }
\|\ul\|_{Z[\ill]}\|\va\|_{\B[\iaa]}. \een And we can still sum up
the frequencies.

%
%
%
\noindent{\textbf{Case 2.2:}} High-high frequency interaction
giving out equal or lower frequency, $\lambda\lesssim\mu$. When
$\lambda\ll\mu$, we cut up intervals as in case (1.2). When
$\lambda\approx\mu$, this procedure degenerate.
\[\lambda\el \um\vm=\sum_{i}\lambda \um^i\vm^i,\hspace{0.3in} \um^i,\vm^i\in X_\mu[\imm^i]\]
Here we don't have a good $L^2$ bound on the product, so we need
to do modulation analysis again to get better control. Also, all
the estimates here have the corresponding version with spatial
localization, the proofs are exactly the same.

\noindent{\textbf{Case 2.2.1:}} $\B\times\B$, both $\um^i$,$\vm^i$
have modulation $\mu^{4+\frac{3}{2s}}\lesssim\sigma\lesssim\mu^3$.
we use Berstein inequality for frequency on product, for
modulation on any one of input. And we have $l^2$ summation of the
small intervals. \ben \|\sum_{i}\lambda (Q_\sigma\um^i)
\vm^i\|_{\yl[\ill]}
&\lesssim&\lambda^{1+s}\Il^{\frac{1}{2}}|\frac{\ill}{\imm}|^{\frac{1}{2}}\sup_i\|(Q_\sigma\um^i )\vm^i\|_{\txt[\imm^i]}\\
&\lesssim&
\lambda^{1+s}\Il\Imu^{-\frac{1}{2}}\lambda^{\frac{1}{2}}\sigma^{\frac{1}{2}}\sup_i\|
\um^i\|_{\txt}\|  \vm^i\|_{\txt}
\\
&\lesssim&\lambda^{5s+\frac{9}{2}}\mu^{-5s-5}\sup_{\imm}\|\um\|_{S[\imm]}\|\vm\|_{X^s}.
\een \noindent{\textbf{Case 2.2.2:}} $\B\times Z$, suppose $\vm$
has modulation $\sm\gtrsim\mu^3$. By modulation analysis
(\ref{HiMode}), this forces another high modulation  on the
output.
\[\lambda\el \um\vm=\sum_{i}\lambda Q_{\approx \sm}[\um^i (Q_{\sm}\vm^i)]\]
We comment that when $\sm\approx\mu^3$, there is chance high
modulation can also fall on $\um$. But in that case, from
Prop~\ref{functionspaceremark}(\ref{modulationmatch}), the norm
$Z$ and $\B$ match with each other. So it is essentially the same
as in the following case 2.2.3.

We use $L^2$~(\ref{L2B}) on $\um$  , and $L^p$ for $\vm$, together
with Bernstein. \ben
&&\|\sum_{i}\lambda Q_{\sm}[\um^i (Q_{\sm} \vm^i)]\|_{\DC[\ill]}\\
&\lesssim&
\lambda^{-2-4s}\sm^{2s+1}|\frac{\ill}{\imm}|^{\frac{1}{2}}\sup_{i}\|\um^i (Q_{\sm} \vm^i)\|_{\txt[\imm]}\\
&\lesssim&\lambda^{-2-4s}\sm^{2s+1}|\frac{\ill}{\imm}|^{\frac{1}{2}}(\lambda\sm)^{\frac{1}{p}}\sup_i
\| \um^i\|_{\txt[\imm^i]}\|Q_{\sm} \vm^i)\|_{L^p_{t,x}[\imm^i]}\\
&\lesssim&\lambda^{-2s-\frac{1}{2}+\frac{1}{p}}\mu^{3s+\frac{1}{2}+\frac{4s+3}{p}}\sup_{\imm}\|
\um\|_{S[\imm]}\|\vm\|_{X^s}. \een We also need to bound the $L^p$
component, here we exchange the interval cutoff with modulation
factor and have $l^p$ summation. \ben
&&\|\sum_{i}\lambda Q_{\approx \sigma}(\um^i Q_{\sigma} \vm^i)\|_{\DLP[\ill]}\\
&\lesssim&
\sigma^{-1}\sup_i\| \um^i\|_{\fxt}\|Q_{\sigma}\vm^i\|_{L^p_{x,t}[\imm^i]}\\
&\lesssim&\sigma^{-1}\mu^{-2s-1}\mu
\sup_i\|\um^i\|_{\B[\imm^i]}\|\vm^i\|_{\LP[\imm^i]}\\
&\lesssim& \mu^{-2s-3 }\sup_{\imm}\|\um\|_{S[\imm]}\|\vm\|_{X^s}.
\een In both case, we can sum up frequency when $-1\leq
s\leq-\frac{3}{4}$.

\noindent{\textbf{Case 2.2.3:}} $Z\times Z$. When $\um,\vm$ both
have high modulation, we put them in $L^3$ (\ref{l3z}).

We begin with the $L^2$ estimate \ben\|\sum_{i}\lambda Q_{\sigma}(
\um^i \vm^i)\|_{\txt[\ill]}
&\lesssim&\lambda(\lambda\sigma)^{\frac{1}{6}}\|\sum_{i}  Q_{\sigma}(\um^i \vm^i)\|_{L^{\frac{3}{2}}_{t,x}[\ill]}\\
&\lesssim&\lambda(\lambda\sigma)^{\frac{1}{6}}|\frac{\ill}{\imm}|^{\frac{2}{3}}\sup_i\|\um^i\|_{L^3_{t,x}[\imm^i]}\|\um^i\|_{L^3_{t,x}[\imm^i]}\\
&\lesssim&\sigma^{\frac{1}{6}}\lambda^{\frac{19}{6}+\frac{8s}{3}}\mu^{-\frac{16}{3}(s+1)
} \sup_{\imm}\|\um\|_{Z[\imm]}\|\vm\|_{X^s}. \een Here notice we
used $l^{\frac{3}{2}}$ summation of the intervals.

From Lemma~\ref{combinecases}, we get \ben \|\sum_{i}\lambda
\um^i\vm^i\|_{\yl[\ill]}
&\lesssim&\lambda^{-\frac{1}{2}+\frac{17(s+1)}{3}}\mu^{-\frac{16}{3}(s+1)
} \sup_{\imm}\|\um\|_{Z[\imm]}\|\vm\|_{X^s}. \een To see we can
sum up frequency, notice exponent for $\mu$ is negative and all
the exponents add up to $-\frac{1}{2}+\frac{1}{3}(s+1) <0$.
\subsection{Estimate for $X^1\times S$.} Suppose $u\in X^1$, $v\in S$. This includes the most dedicate case, i.e. low frequency high modulation piece interact with high frequency
low modulation, where we can not prove the bilinear estimate
(\ref{primaryestimate}). Instead we have to reiterate the equation
and turn the bilinear estimate to trilinear. Let us work on
high-high frequency interaction first.

\noindent{\textbf{Case 3.1:}} High-high frequency interaction
giving out equal or lower frequency, $\lambda\lesssim\mu$. Same as
before, we need to cut into smaller intervals if $\lambda\ll\mu$,
and this procedure degenerate if $\lambda\approx \mu$.

\noindent{\textbf{Case 3.1.1:}} $X^1\times \B$, by (\ref{HiMode})
we must have modulation $\sigma\gtrsim\lambda\mu^2$ in some term.
\[\lambda\el \um\vm=\sum_{i}\sum_{Q_j\in \{Q_{\gtrsim \lambda\mu^2}, Q_{\ll \lambda\mu^2}\}}\lambda Q_1[(Q_2 \um^i)(Q_3 \vm^i)].\]
\noindent{\textbf{Case 3.1.1(a):}} When high modulation is on
output, i.e. $Q_1=Q_{\sigma\gtrsim\lambda\mu^2}$. Using
$L^\infty_tL^2_x$ on $\um^i$, $\txt$ on $\vm^i$, together with
Bernstein on the product, we get, \ben \|\lambda
\sum_{i}Q_{\sigma}[(Q_2 \um^i)(Q_3 \vm^i)]\|_{\txt[\ill]}
&\lesssim&\lambda^{\frac{3}{2}}|\frac{\ill}{\imm}|^{\frac{1}{2}}\sup_i\| \um^i\|_{L^\infty_tL^2_x[\imm^i]}\|\vm^i\|_{\txt[\imm^i]}\\
&\lesssim&\lambda^{2s+3}\mu^{-4s-\frac{7}{2}}\sup_{\imm}\|\um\|_{X^1[\imm]}\|\vm\|_{X^s}.
\een Using the fact that output has high modulation and
Lemma~\ref{combinecases} with $k=\frac{1}{2}$, we get \ben
\|\lambda
\sum_{i}Q_{\sigma\gtrsim\lambda\mu^2}(\um^i\vm^i)\|_{DZ[\ill]}
&\lesssim&\lambda^{4s+4}\mu^{-4s-\frac{9}{2}}\sup_{\imm}\|\um\|_{X^1[\imm]}\|\vm\|_{X^s}.
\een
%
%
\noindent{\textbf{Case 3.1.1(b):}} High modulation on $\um$,
$Q_2=Q_{\gtrsim\lambda\mu^2}$. We put them both in $\txt$. \ben
&&\|\sum_{i}\lambda Q_{\sigma}[(Q_{\gtrsim\lambda\mu^2}\um^i)Q_3 \vm^i]\|_{\txt[\ill]}\\
&\lesssim&\lambda(\lambda\sigma)^{\frac{1}{2}}\|\sum_{i}  Q_{\sigma}[(Q_{\gtrsim\lambda\mu^2}\um^i)Q_3 \vm^i]\|_{L^1_{t,x}[\ill]}\\
&\lesssim&\lambda (\lambda\sigma)^{\frac{1}{2}}|\frac{\ill}{\imm}|\sup_i\|Q_{\gtrsim\lambda\mu^2} \um^i\|_{\txt[\imm^i]}\|Q_3  \vm^i\|_{\txt[\imm^i]}\\
&\lesssim&\sigma^{\frac{1}{2}}\lambda^{4s+\frac{7}{2}}\mu^{-8s-\frac{17}{2}}\sup_{\imm}\|\um\|_{X^1[\imm]}\|\vm\|_{X^s}.
\een hence we have \ben \|\sum_{i}\lambda
(Q_{\gtrsim\lambda\mu^2}\um^i)Q_3
\vm^i\|_{\yl[\ill]}&\lesssim&\lambda^{7s+\frac{13}{2}}\mu^{-8s-\frac{17}{2}}\sup_{\imm}\|\um\|_{X^1[\imm]}\|\vm\|_{X^s}.
\een

\noindent{\textbf{Case 3.1.1(c):}} High modulation comes from
input $Q_3=Q_{\gtrsim\lambda\mu^2}$. We use local smoothing
(\ref{localsmoothing}) on $\um$, and $\txt$ on $\vm$. \ben
&&\|\sum_{i}\lambda Q_{\sigma}[Q_{2} \um^i (Q_{\gtrsim\lambda\mu^2} \vm^i)]\|_{\txt[\ill]}\\
&\lesssim&\lambda \sigma^{\frac{1}{2}}\|\sum_{i} Q_{\sigma}[Q_{2} \um^i (Q_{\gtrsim\lambda\mu^2} \vm^i)]\|_{L^2_xL^1_t[\ill]}\\
&\lesssim& \lambda
\sigma^{\frac{1}{2}}|\frac{\ill}{\imm}|\sup_i\|Q_{2}
\um^i\|_{L^{\infty}_xL^2_t[\imm^i]}
\|Q_{\gtrsim\lambda\mu^2}  \vm^i\|_{\txt[\imm^i]}\\
&\lesssim&\sigma^{\frac{1}{2}}\lambda^{3s+3}\mu^{-6s-6}\sup_{\imm}\|\um\|_{X^1[\imm]}\|\vm\|_{X^s}.
\een Hence we have \ben \|\sum_{i}\lambda  Q_{2} \um^i
(Q_{\gtrsim\lambda\mu^2} \vm^i)\|_{\yl[\ill]}&\lesssim&
\lambda^{6s+6}\mu^{-6s-6}\|\um\|_{X^1[\imm]}\|\vm\|_{S[\imm]}.
\een

\noindent{\textbf{Case 3.1.2:}} $X^1\times Z$. This forces high
modulation $\sm\gtrsim\mu^3$ also on the output.
\[\lambda\el \um\vm=\sum_{i}\lambda Q_{\approx \sm}[ \um^i (Q_{\sm}\vm^i)]\]
We still bound the output in $L^2$ by using  $L^6$ on $\um$, $L^3$
(\ref{l3z}) on $\vm$. \ben
&& \| \sum_{i}\lambda Q_{\sigma}[ \um^i (Q_{\sm}\vm^i)]\|_{\txt[\ill]}\\
&\lesssim&\lambda |\frac{\ill}{\imm}|^{\frac{1}{2}}\sup_i\| \um^i\|_{L^6_{t,x}[\imm^i]}\|Q_{\sm} \vm^i\|_{L^3_{t,x}[\imm^i]}\\
&\lesssim&\lambda^{2s+\frac{5}{2}}
\mu^{-2s-\frac{3}{2}}\mu^{-\frac{1}{6}-s}\mu^{
\threez}\sup_i\|\um^i\|_{X^1[\imm^i]}\|\vm^i\|_{S[\imm^i]}\\
&\lesssim&\lambda^{2s+\frac{5}{2}}\mu^{-5s-4+\frac{2}{3}(s+1)
}\sup_{\imm}\|\um\|_{X^1[\imm]}\|\vm\|_{X^s}. \een From
Lemma~\ref{combinecases} with $k=\frac{1}{2}$, we get \ben\|
\sum_{i}\lambda Q_{\sm}[ \um^i
(Q_{\sm}\vm^i)]\|_{DZ[\ill]}\lesssim
\lambda^{4s+4}\mu^{-5s-\frac{11}{2}+\frac{2}{3}(s+1)}\sup_{\imm}\|\um\|_{X^1[\imm]}\|\vm\|_{X^s}.
\een \noindent{\textbf{Case 3.2:}} High low frequency interaction.
$\ua\in X^1, \vl\in S$, $\lambda\gg\alpha$. The bilinear estimate
(\ref{L2SSX}) is not good enough, so we have to break into more
cases.

\noindent{\textbf{Case 3.2.1:}} $\ua\in X^1, \vl\in \B$. Because
of high modulation relation (\ref{HiMode}), we have
\[\lambda \el \ua\vl=\sum_{Q_i\in
\{Q_{\gtrsim \lambda^2\alpha}, Q_{\ll
\lambda^2\alpha}\}}\lambda\el Q_1[(Q_2\ua)(Q_3 \vl)].\]

\noindent{\textbf{Case 3.2.1(a):}} High modulation  on $\ua$.
$Q_2=Q_{\sigma\gtrsim\lambda^2\alpha}$. Put $\ua$ in $L^2$, $\vl$
in $L^\infty$ (\ref{LinftyB}). \ben \|\lambda\el
Q_1[(Q_{\sigma\gtrsim\lambda^2\alpha}\ua)(Q_3\vl)]\|_{\txt[\ill]}\lesssim
\lambda^{-4s-\frac{7}{2}}\alpha^{-1-s}\|\ua\|_{X^1[\iaa]}\|\vl\|_{\B[\ill]},
\een so from Lemma \ref{combinecases}, we get \ben \|\lambda\el
[(Q_{\sigma\gtrsim\lambda^2\alpha}\ua)(Q_3
\vl)]\|_{\yl[\ill]}\lesssim
\lambda^{-s-2}\alpha^{-1-s}\|\ua\|_{X^1[\iaa]}\|\vl\|_{\B[\ill]}.
\een \noindent{\textbf{Case 3.2.1(b):}} High modulation   on
$\vl$. $Q_3=Q_{\sigma\gtrsim\lambda^2\alpha}$. Put $\ua$ in
$L^\infty$, $\vl$ in $L^2$. \ben \|\lambda\el
Q_1[(Q_2\ua)(Q_{\sigma\gtrsim\lambda^2\alpha}\vl)]\|_{\txt[\ill]}\lesssim
\lambda^{-1-s}\alpha^{-\frac{1}{2}-2s}\|\ua\|_{X^1[\iaa]}\|\vl\|_{\B[\ill]},
\een so we get \ben \|\lambda\el [(Q_2
\ua)(Q_{\sigma\gtrsim\lambda^2\alpha}\vl)]\|_{\yl[\ill]}\lesssim
\lambda^{2s+\frac{1}{2}}\alpha^{-\frac{1}{2}-2s}\|\ua\|_{X^1[\iaa]}\|\vl\|_{\B[\ill]}.
\een \noindent{\textbf{Case 3.2.1(c):}} When none of $\ua,\vl$
have high modulation, this forces the output to be approximately
$\lambda^2\alpha$. $Q_1=Q_{\sigma\approx\lambda^2\alpha}$, put
$\ua$ in $L^\infty$, $\vl$ in $L^2$.

When $\lambda^2\alpha\lesssim\lambda^{4+\frac{3}{2s}}$, i.e.
$\alpha\lesssim\lambda^{2+\frac{3}{2s}}$ we have \ben
&&\|\lambda\el Q_{\sigma\approx\lambda^2\alpha}[(Q_2 \ua)(Q_{3} \vl)] \|_{\DI}\\
&\lesssim&\lambda^{1+s}|\ill|^{\frac{1}{2}}
\alpha^{\frac{1}{2}-s}\lambda^{-2-s}\|\ua\|_{X^1[\iaa]}\|\vl\|_{\B[\ill]}
\\
&\lesssim&\alpha^{\frac{1}{2}-s}\lambda^{2s+\frac{1}{2}}\|\ua\|_{X^1[\iaa]}\|\vl\|_{\B[\ill]},
\een notice we have
$\alpha^{\frac{1}{2}-s}\lambda^{2s+\frac{1}{2}}\lesssim\lambda^{\frac{3}{4s}}$,
which is good for summation.

When $\lambda^2\alpha\gtrsim\lambda^{4+\frac{3}{2s}}$,  we have
\ben
&&\|\lambda\el Q_{\sigma\approx\lambda^2\alpha}[(Q_2\ua)(Q_{3}\vl)] \|_{\DB}\\
&\lesssim&\lambda^{1-s}\sigma^s
\alpha^{\frac{1}{2}-s}\lambda^{-2-s}\|\ua\|_{X^1[\iaa]}\|\vl\|_{\B[\ill]}
\\
&\lesssim&\alpha^{\frac{1}{2}}\lambda^{-1}\|\ua\|_{X^1[\iaa]}\|\vl\|_{\B[\ill]}.
\een

\noindent{\textbf{Case 3.2.2:}} $\ua\in X^1, \vl\in Z$. Here the
bilinear estimate (\ref{L2SX}) is good enough. \ben
\|\lambda\el\ua\vl\|_{\DI[\ill]}&\lesssim&\lambda^{1+s}|\ill|^{\frac{1}{2}}\|\el\ua\vl\|_{\txt[\ill]}
\\
&\lesssim&
\lambda^{1+s}|\ill|^{\frac{1}{2}}\lambda^{-\frac{1}{3}-2s-2}\alpha^{-\frac{1}{6}-s}\|\ul\|_{\zll}\|\va\|_{\xaa}\\
&\lesssim&\lambda^{\frac{1}{6}+s}\alpha^{-\frac{1}{6}-s}\|\ul\|_{\zll}\|\va\|_{\xaa}.
\een

\noindent{\textbf{Case 3.3:}} High low frequency interaction.
$\ul\in X^1, \va\in S$, $\lambda\gg\alpha$.

\noindent{\textbf{Case 3.3.1:}} $\ul\in X^1$, $\va\in \B$. Without
going into modulation analysis, we use $L^\infty_xL^2_t$ on $\ul$,
and
 $L^2_xL^\infty_t$ on $\va$, together with Bernstein and notice the modulation on $\va$ is small.
\ben
\|\lambda\el\ul\va\|_{\DI}&\lesssim& \lambda^{1+s}|\ill|^{\frac{1}{2}}\|\vl\|_{L^\infty_xL^2_t[\ill]}\|\va\|_{L^2_xL^\infty_t[\iaa]}\\
&\lesssim& \lambda^{1+s}|\ill|^{\frac{1}{2}}\lambda^{-1-s}\alpha^{-\frac{3}{2}-2s}\|\ul\|_{X^1[\ill]}\|\va\|_{\B[\iaa]}\\
&\lesssim&\lambda^{2s+\frac{3}{2}}\alpha^{-\frac{3}{2}-2s}\|\ul\|_{X^1[\ill]}\|\va\|_{\B[\iaa]}.
\een
%
%
%
\noindent{\textbf{Case 3.3.2:}} $\ul\in X^1$, $\va\in Z$. Here we
can not prove any bilinear estimate if high modulation fall on
$\va$, so we need the following lemma to reiterate the equation.

\begin{lemma}(Reiterate the equation)\label{backflow} Let $w$ be a solution to $KdV$
equation \textnormal{(\ref{kdv})}. Then we can write its high
modulation part as
\[Q_{\sigma\gtrsim\alpha^3}w_{\alpha}=M_1+ M_2+R,\]
where $M_1, M_2, R$ are as follows:
\begin{itemize}
\item$M_1$ is the output  of two higher frequency-low modulation
interaction,
\[M_1=\sum_{\alpha\lesssim\beta_1\approx\beta_2}\LL^{-1}\alpha
P_\alpha Q_\sigma({w}_{\beta_1}w_{\beta_2}),  \mbox{~~}
w_{\beta_1}, w_{\beta_2}\in X^1\] where $w_{\beta_1}, w_{\beta_2}$
all have very low modulation
$|\tau-\xi^3|\lesssim|\xi|^{4+\frac{3}{2s}}$. \item
 $M_2$ is the output of the high
frequency-low modulation piece interact with low frequency-high
modulation piece.
\[M_2=\sum_{\sigma\gtrsim\alpha^3,\gamma\ll\beta\approx\alpha}\LL^{-1}\alpha P_\alpha Q_\sigma(w_{\beta}w_{\gamma}),  \mbox{~~}  w_{\beta}\in X^1, w_{\gamma}\in Z.\]
$w_\beta$ has   modulation
$|\tau-\xi^3|\lesssim|\xi|^{4+\frac{3}{2s}}$, $w_\gamma$ has high
modulation $|\tau-\xi^3|\gtrsim|\xi|^3$. \item $R$ is the
remainder, which comes from interaction of all other cases
\[R=\sum_{\sigma\gtrsim\alpha^3, \beta,\gamma}\LL^{-1}\alpha
P_\alpha Q_\sigma(w_{\beta}w_{\gamma}).\] For $R$, we have the
estimate
\begin{equation}\|\eta_{\alpha}(t)R_\alpha\|_{\alpha^{-2s-\frac{3}{2}}L^2_xL^{\infty}_t}\lesssim
\|w\|^2_{X^s\cap X^s_{le}[\iaa]}.\label{Rest}\end{equation}
\end{itemize}
The decomposition above is true modulo $\pm$ sign on each term.
\end{lemma}
\begin{proof}: If we apply frequency and modulation projection on the equation, we get
\[(\partial_t+\partial^3_x)P_\alpha Q_{\sigma}w=-P_\alpha
Q_{\sigma}\partial_x(w^2).\] Hence modulo $\pm$ sign we have
\[P_\alpha Q_{\sigma}w=\LL^{-1}\alpha P_\alpha Q_{\sigma}(w^2).\]
Here we decompose $w$ into dyadic pieces, $P_\alpha
Q_{\sigma}w=\LL^{-1}\alpha P_\alpha
Q_{\sigma}(w_{\beta}w_{\gamma})$. Now we first break  each
$w_\lambda$  into sum of functions supported on time scale
$|\lambda|^{4s+3}$. Next, for each $w_\lambda\in \xx[\ill]$, let
us  decompose it as  $w_\lambda=w_{\lambda,1}+w_{\lambda,2}$,
$w_{\lambda,1}\in X^1$,  $w_{\lambda,2}\in S$. Then we can just
take $\ub, \vr$ to represent $w_{\beta,i},\hspace{0.05in}
w_{\gamma,j}$, \hspace{0.05in}$i,j\in\{1,2\}$.

We will prove  that except for the two cases in $M_1$ and $M_2$,
we have the estimate (\ref{Rest}).
We list the estimates of all cases below, which are similar to what we have done before. Notice the modulation is  always larger than $\alpha^3$ in the summation.  \\
\noindent\textbf{\textbf{Case 1}}: $\beta\approx\alpha\gg\gamma$.

\noindent\textbf{(1)} $\ua,\vr\in X^1$, use Bernstein and bilinear
estimate (\ref{L2XX}) \ben
\|\ea\sum_\sigma\LL^{-1}\alpha\pq(\ua\vr)\|_{\rs}&\lesssim&\frac{\alpha^{2s+\frac{5}{2}}}
{\sigma^\frac{1}{2}}\alpha^{-1-s}\gamma^{-s}\|\ua\|_{X^1[\iaa]}\|\vr\|_{X^1[I_\gamma]}\\
&\lesssim&\alpha^{-1+s}
\gamma^{-s}\|\ua\|_{X^1[\iaa]}\|\vr\|_{X^1[I_\gamma]}. \een
\textbf{(2) }$\ua\in X^1, \vr\in S$  we only deal with $\vr\in
\B$. And leave $\vr\in Z$ term into $M_2$. Notice here $\ua$ must
have high modulation $\sigma$ (\ref{HiMode}). Put $L^2$ on $\ua$,
$L^\infty$ on $\vr$ \ben
&&\|\ea\sum_\sigma\LL^{-1}\alpha\pq((Q_\sigma\ua)\vr)\|_{\rs}
\\
&\lesssim&\alpha^{2s+\frac{5}{2}}\sigma^{-\frac{1}{2}}
\sigma^{-1}|\iaa|^{-\frac{1}{2}}\alpha^{-s}\gamma^{-2s-1}\|\ua\|_{X^1[\iaa]}\|\vr\|_{\B[I_\gamma]}\\
&\lesssim&\alpha^{-s-\frac{7}{2}}
\gamma^{-2s-1}\|\ua\|_{X^1[\iaa]}\|\vr\|_{\B[I_\gamma]}. \een
\textbf{(3)} $\ua\in S, \vr\in X^1$,  use the bilinear estimate
(\ref{L2SSX}) \ben
&&\|\ea\sum_\sigma\LL^{-1}\alpha\pq(\ua\vr)\|_{\rs}\\
&\lesssim&
\alpha^{2s+\frac{5}{2}}\sigma^{-\frac{1}{2}}\max{\{\alpha^{-\frac{1}{3}-2s-2}\gamma^{-\frac{1}{6}-s},\alpha^{-2-s}\gamma^{\frac{1}{2}-s}\}}
\|\ua\|_{S[\iaa]}\|\vr\|_{X^1[I_\gamma]}\\
&\lesssim&\max{\{\alpha^{-\frac{4}{3}}\gamma^{-\frac{1}{6}-s},\alpha^{s-1}
\gamma^{\frac{1}{2}-s}\}}\|\ua\|_{S[\iaa]}\|\vr\|_{X^1[I_\gamma]}.
\een
\textbf{(4)} $\ua,\vr\in S$, we consider several cases:
\\
If $\ua, \vr\in\B$, or $\ua\in Z, \vr\in\B$, or $\ua,\vr\in Z$,
then we have the bilinear estimate
 (\ref{L2SSS}). So we have
\ben
&&\|\ea\sum_\sigma \LL^{-1}\alpha\pq((Q_\sigma\ua)\vr)\|_{\rs}\\
&\lesssim&\frac{\alpha^{2s+\frac{5}{2}}}
{\sigma^{\frac{1}{2}}}\alpha^{\threez}\gamma^{\sixz}\|\ua\|_{\B[\iaa]}\|\vr\|_{\B[I_\gamma]}\\
&\lesssim&\alpha^{-\frac{4}{3}+\tail}
\gamma^{\sixz}\|\ua\|_{\B[\iaa]}\|\vr\|_{\B[I_\gamma]}. \een
Notice the exponents add up to $-1.$

If $\ua\in \B, \vr\in Z$, use $L^2$ on $\ua$, $L^p$ on $\vr$ \ben
&&\|\ea\sum_\sigma\LL^{-1}\alpha\pq(\ua\vr)\|_{\rs}\\
&\lesssim& \alpha^{2s+\frac{5}{2}} \sigma^{-1} \sigma^{\frac{1}{2} }\alpha^{-2-s}\gamma^{1 }\|\ua\|_{\B[\iaa]}\|\vr\|_{Z[I_\gamma]}\\
&\lesssim&\alpha^{s-1
}\gamma\|\ua\|_{\B[\iaa]}\|\vr\|_{\B[I_\gamma]}. \een The
exponents add up to $s<0$.

\noindent{\textbf{Case 2}}: $\beta\approx \gamma\gtrsim \alpha$.
This part is every similar to the estimates in Case 1.2, 2.2  and
3.1. We still need to decompose $\ub$ into sums of functions that
are supported on the $\mu^{4s+3}$ time scale. $\ub=\sum_i \ub^i,
\hspace{0.1in}\ub^i\in X_\beta[\ibb^i]$

\noindent\textbf{(1)} $\ub,\vb\in X^1$, when one of input e.g.
$\ub$ has high modulation $Q_{\gtrsim \alpha\beta^2}$, estimate
$\ub$ in $L^2$, and $\vb$ in $L^\infty_xL^2_t$. Here because we
want to use Bernstein, but also want to have better summation of
time intervals. So we need to use local energy space $X^s_{le}$
similarly as in case 1.2(b).
\begin{eqnarray*}&&\sum_j\|\chi_j^\alpha(x)\ea\sum_{\sigma}\LL^{-1}\alpha\pq ((\Qm
\ub)\vb)\|^2_{\rs}\\
&\lesssim&\sum_j \alpha^{4s+5}\|\chi_j^\alpha(x)\sum_{\sigma}\pq
((\Qm
\ub)\vb)\|^2_{L^2_xL^1_t[\iaa]}\\
&\lesssim& \sum_j\alpha^{4s+5}\|\chi_j^\alpha(x)Q_{\sm}\ub\|^2_{\txt[\iaa]}\|\chi_j^\alpha(x)\vb\|^2_{L^{\infty}_xL^2_t[\iaa]}\\
&\lesssim&\alpha^{4s+5}\sum_{i,j}\|\chi_j^\alpha(x)Q_{\sm}\ub^i\|^2_{\txt[\ibb^i]}\sup_j\sum_i \|\chi_j^\alpha(x)\vb^i\|^2_{L^{\infty}_xL^2_t[\ibb^i]}\\
&\lesssim&\alpha^{4s+5}|\frac{\iaa}{\ibb}|\sup_i\|Q_{\sm}\ub^i\|_{\txt[\ibb^i]}\sup_j\sum_i \|\chi_j^\alpha(x)\vb^i\|^2_{L^{\infty}_xL^2_t[\ibb^i]}\\
&\lesssim&\alpha^{8s+6}\beta^{-12s-12}\|\ub\|^2_{X_\beta[\ibb]}\|\vb\|^2_{X^s_{le}
}.
\end{eqnarray*}
\begin{remark} In   these estimates, we need to sum up all the modulations larger than $\alpha^3$. It is fine as long as there is a negative factor of $\sigma$
through the estimate. But in the one above, we need be more
careful. Split the problem into $\sigma\approx\alpha^3$, and
$\sigma\gg\alpha$.

When $\sigma\approx \alpha^3$, we can sum up modulation easily.

When $ \sigma\gg\alpha^3$, we prove $\LL^{-1} :
L^2_xL^1_t\rightarrow L^2_xL^\infty_t$ is bounded operator£¬ which
is done by looking at the symbol
$\frac{1}{\tau-\xi^3}=\frac{1}{\tau}+\frac{\xi^3}{\tau(\tau-\xi^3)}\approx
\frac{1}{\tau}$. And
 $\partial_t^{-1}: L^2_xL^1_t\rightarrow L^2_xL^\infty_t$  is bounded if it acts on functions which vanish at $\infty$.
\end{remark}
\noindent\textbf{(2)} $\ub\in X^1,\vb\in S$, we also split it into
two cases:

\noindent (a) When $\ub\in X^1,\vb\in \B$. Now if the output
modulation $\sigma\gtrsim\alpha\beta^2$, use $L^\infty$ on $\ub$,
and $L^2$ on $\vb$,
\ben&&\|\ea\sum_{\sigma\gtrsim\alpha\beta^2}\LL^{-1}\alpha
 \sum_{i}\pq
(\ub^i\vb^i)\|_{\rs}\\&\lesssim&
\frac{\alpha^{2s+\frac{5}{2}}}{\sigma^{\frac{1}{2}}}
|\frac{\iaa}{\ibb}|^{\frac{1}{2}}\beta^{\frac{1}{2}-s}\beta^{-2-s}\sup_i\|\ub^i\|_{\xbb^i}\|\vb^i\|_{\B[\ibb^i]}\\
&\lesssim&\alpha^{4s+\frac{7}{2}}\beta^{-4s-4}\|\ub\|_{X^s}\|\vb\|_{X^s}
\een And if $\sigma\ll\alpha\beta^2$, we use $L^\infty_xL^2_t$ on
$\ub$, $L^2$ on $\vb$. We still play the trick: using local energy
space to get $l^2$ summation of the intervals. \ben&&\|\ea
\sum_{\sigma{\ll\alpha\beta^2}}\LL^{-1}\alpha\sum_{i}\pq
(\ub^i\vb^i)\|_{\rs}\\
&\lesssim& \alpha^{2s+\frac{5}{2}}
|\frac{\iaa}{\ibb}|^{\frac{1}{2}}\beta^{-1-s}\beta^{-2-s}\sup_i\|\ub^i\|_{X^1[\ibb^i]}\|\vb^i\|_{X^s_{le}}\\
&\lesssim&\alpha^{4s+4}\beta^{-4s-\frac{9}{2}}\|\ub
\|_{X^s}\|\vb\|_{X^s_{le}}. \een The point here is we can sum up
the modulation $\alpha^3\lesssim\sigma\ll\alpha\beta^2$, which
give us at most $\log\beta$ loss. But we are fine because of the
negative power on $\beta$. We will do a similar thing whenever we
want to be careful with modulation summation, hence we will ignore
it.

\noindent(b) When $\ub\in X^1,\vb\in Z$. This force high
modulation $\sm\gtrsim\beta^3$  on $\ub$, or on output.

When $\sm$ is on $\ub$, use $L^2$ on $\ub$, $L^\infty$ on $\vb$.
\ben&&\|\ea \sum_{\sigma }\LL^{-1}\alpha\sum_{i}\pq
(\ub^i\vb^i)\|_{\rs}\\&\lesssim&
\frac{\alpha^{2s+\frac{5}{2}}}{\sigma}\sigma^{\frac{1}{2}}
|\frac{\iaa}{\ibb}|^{\frac{1}{2}
}\sm^{-1}|\ibb|^{-\frac{1}{2}}\beta^{-s}\beta
\sup_i\|\ub^i\|_{X^1[\ibb^i]}\|\vb^i\|_{Z[\ibb^i]}\\
&\lesssim&\alpha^{4s+\frac{5}{2} }\beta^{-5s-5
}\|\ub\|_{X^s}\|\vb\|_{X^s}. \een When $\sm$ is on output, simply
put $L^6$ on $\ub$, and $L^3$ on $\vb$. \ben&&\|\ea\sum_{\sigma
}\LL^{-1}\alpha\sum_{i}\pq (\ub^iQ_{\sm}\vb^i)\|_{\rs}\\&\lesssim&
\frac{\alpha^{2s+\frac{5}{2}}}{\sm}\sm^{\frac{1}{2}}|\frac{\iaa}{\ibb}|^{\frac{1}{2}}\beta^{-\frac{1}{6}-s}\beta^{\threez}
\sup_i\|\ub^i\|_{X^1[\ibb^i]}\|\vb^i\|_{Z[\ibb^i]}\\
&\lesssim&\alpha^{4s+\frac{7}{2}}\beta^{-5s-5-\tail}\|\ub\|_{X^s}\|\vb\|_{X^s}.
\een
\textbf{(3)} $\ub,\vb\in S$. We still break into cases.

\noindent \textbf{(a)} $\ub, \vb\in \B$, use $L^2$ on both,  and
$l^1$ summation of interval is good enough.
\begin{eqnarray*}&&\|\ea\sum_{\sigma }\LL^{-1}\alpha\sum_{i}\pq
(\ub^i\vb^i)\|_{\rs}\\&\lesssim&\alpha^{2s+3}|\frac{\iaa}{\ibb}|\sup_i\| \ub^i\|_{\txt[\ibb^i]}\|\vb^i\|_{\txt[\ibb^i]}\\
&\lesssim&\alpha^{6s+6}\beta^{-6s-7}\|\ub\|_{X^s}\|\vb\|_{X^s}.
\end{eqnarray*}
\noindent \textbf{(b)} $\ub\in \B, \vb\in Z$, $L^2$ on $\ub$,
$L^3$ on $\vb$, with a $l^1$ summation of interval.
\begin{eqnarray*}&&\|\ea\sum_{\sigma }\LL^{-1}\alpha\sum_{i}\pq
(\ub^i\vb^i)\|_{\rs}\\&\lesssim&\alpha^{2s+\frac{5}{2}+\frac{1}{3}}\sigma^{-\frac{1}{6}}|\frac{\iaa}{\ibb}|
\sup_i\| \ub^i\|_{\txt[\ibb^i]}\|\vb^i\|_{L^3_{t,x}[\ibb^i]}\\
&\lesssim&\alpha^{6s+\frac{16}{3}}\beta^{-7s-7-\frac{1}{3}+\tail}\|\ub\|_{X^s}\|\vb\|_{X^s}.
\end{eqnarray*}

\noindent \textbf{(c)} $\ub , \vb\in Z$, Here we are a bit careful
about interval cut off, using the $l^{\frac{3}{2}}$ summation.
\begin{eqnarray*}&&\|\ea\sum_{\sigma }\LL^{-1}\alpha\sum_{i}\pq
(\ub^i\vb^i)\|_{\rs}\\
&\lesssim&\alpha^{2s+\frac{5}{2}+\frac{1}{6}}\sigma^{-\frac{1}{3}}\|\sum_{i}\pq
(\ub^i\vb^i)\|_{L^{\frac{3}{2}}_{t,x}[\iaa]}\\
&\lesssim&\alpha^{2s+\frac{5}{2}+\frac{1}{6}}\sigma^{-\frac{1}{3}}|\frac{\iaa}{\ibb}|^{\frac{2}{3}}
\sup_i\|\ub^i\|_{L^3_{t,x}[\ibb^i]}\|\vb^i\|_{L^3_{t,x}[\ibb^i]}\\
&\lesssim&\alpha^{2s+1+\frac{8s+8}{3}}\beta^{-\frac{16}{3}(s+1) }
\|\ub\|_{X^s}\|\vb\|_{X^s}.
\end{eqnarray*}
\end{proof}
  \noindent Now we use this lemma to finish our
estimate of Case 3.3,  $\ul\in X^1, \va\in Z$.
\[\lambda \ul\va=\lambda \ul (M_{1\alpha}+M_{2\alpha}+R_\alpha).\]
\textbf{Step 1:}  Let us do  $R_\alpha$ first, using the estimate
for $R_\alpha$
 in the lemma.
\begin{eqnarray*}
 \|\lambda \el \ul R_\alpha\|_{\DI}
&\lesssim& \lambda^{3s+\frac{5}{2}}\|\el\ul\|_{\fxtt}\|\el
R_{\alpha}\|_{L^2_xL^{\infty}_t}\\
&\lesssim&\lambda^{3s+\frac{5}{2}}\lambda^{-1-s}\alpha^{-2s-\frac{3}{2}}\|\el\ul\|_{X^1}\|\el
R_{\alpha}\|_{\rs}\\
&\lesssim&(\frac{\alpha}{\lambda})^{-2s-\frac{3}{2}}\|u_{\lambda}\|_{X^s[\ill]}\|v\|^2_{X^s\cap
X^s_{le}}.
\end{eqnarray*}
\textbf{Step 2:} Feed $M_1$ into the bilinear term, we divide it
into two terms. \ben
\lambda \ul \hspace{-0.2in}\sum_{\sigma \approx \alpha\beta^2,
\alpha\lesssim\beta\ll\lambda}\hspace{-0.2in}\LL^{-1}\alpha
{\pq}_{\approx\alpha\beta^2}(\vb \vb) +\lambda \ul
\hspace{-0.2in}\sum_{\sigma \approx \alpha\beta^2,
\alpha\lesssim\lambda\lesssim\beta}\hspace{-0.2in}\LL^{-1}\alpha
 {\pq}_{\approx\alpha\beta^2}(\vb \vb) \een
The first term, we will bilinear estimate for $\ub\vl$,
also here    for fixed $\beta$, $P_{\alpha}Q_{\alpha\beta^2}(\vb
\vb)$ is almost althogonal to each other, so we can sum up
$\alpha$
\begin{eqnarray*}
 &&\|\lambda \el\ul
\sum_{\sigma \approx \alpha\beta^2,
\alpha\lesssim\beta\ll\lambda}\LL^{-1}\alpha
{\pq}_{\approx\alpha\beta^2}(\vb \vb)\|_{\DI}\\
&\lesssim&\lambda^{3s+\frac{5}{2}}\sum_{\alpha\lesssim\beta\ll\lambda}\frac{1}{\beta^2}\|\el
\ul P_{\alpha}Q_{\alpha\beta^2}(\vb
\vb)\|_{\txt}\\
&\lesssim&  \lambda^{3s+\frac{5}{2}}\sum_{\alpha\lesssim\beta\ll\lambda}\frac{1}{\beta^2}\|\el\ul\ub\|_{\txt}\|\el\vb\|_{\fxt}\\
&\lesssim&
\lambda^{3s+\frac{5}{2}}\sum_{\alpha\lesssim\beta\ll\lambda}\frac{1}{\beta^2}\lambda^{-1-s}
\alpha^{\frac{1}{2}}\beta^{-2s}\|\ul\|_{\xll}\|\vb\|_{\xbb}\|\vb\|_{\xbb}\\
&\lesssim&(\frac{\lambda}{\beta})^{2s+\frac{3}{2}}\|\ul\|_{X^s[\ill]}\|\vb\|_{X^s[\ibb]}\|\vb\|_{X^s[\ibb]}.
\end{eqnarray*}
Here we actually used the fact that, when fix $\alpha$, the two
$\vb$'s can be decomposed to functions with $\hat{v}_{\beta}$
supported on size $\alpha$ interval, so we used bernstein to get
 $$\|\vb\|_{\fxt}\lesssim \alpha^{\frac{1}{2}}\|\vb\|_{L^{\infty}_tL^2_x}.$$
So for $s\leq-\frac{3}{4}$, we can sum up $\beta$.

For the second term, we will use at least $l^4$ interval summation
(or better if we use local energy space). The good thing is that
for $\beta$ fixed, then $P_{\alpha}Q_{\alpha\beta^2}(\vb \vb)$ are
almost orthogonal to each other in both space and time, so we can
sum up $\alpha$ and then ignore it. Also because $\ub$ is measured
on the smallest time scale, we still need to cut the interval.
\begin{eqnarray*}
&&\|\lambda\el \ul \sum_{\sigma \approx \alpha\beta^2,
\alpha\lesssim\lambda\lesssim\beta}\LL^{-1}\alpha
 {\pq}_{\approx\alpha\beta^2}(\vb \vb)\|_{\DI}\\
&\lesssim& \lambda^{3s+\frac{5}{2}}\|\el \ul\|_{L^4_t
L^{\infty}_x}\|\el\sum_{\alpha\lesssim\lambda\lesssim\beta}\frac{1}{\beta^2}P_{\alpha}Q_{\alpha\beta^2}(\vb
\vb)\|_{L^4_tL^2_x}\\
&\lesssim&\lambda^{3s+\frac{5}{2}-\frac{1}{4}-s}\|\el \ul\|_{X^1}
\sum_{\lambda\lesssim\beta}\|\el\sum_{\alpha\lesssim\lambda}P_{\alpha}Q_{\alpha\beta^2}(\vb\vb)\|_{L^4_tL^2_x}\\
&\lesssim&\lambda^{2s+\frac{9}{4}}\|\ul\|_{X^s}\sum_{\lambda\lesssim\beta}|\frac{\ill}{\ibb}|^{\frac{1}{4}}
\|\eb(\vb\vb)\|_{L^4_tL^2_x}\\
&\lesssim& \lambda^{2s+\frac{9}{4}}\|\ul\|_{X^s}
\sum_{\lambda\lesssim\beta}|\frac{\ill}{\ibb}|^{\frac{1}{4}}
\|\eb\vb\|_{L^8_tL^4_x}\|\eb\vb\|_{L^8_tL^4_x}\\
&\lesssim&
(\frac{\lambda}{\beta})^{3s+3}\|\ul\|_{X^s}\|\vb\|^2_{X^s}
\end{eqnarray*}

So we combine the two cases together and get $$\|\lambda \ul
\sum_{\alpha\ll\lambda}M_{1\alpha}\|_{\DI}\lesssim\|\ul\|_{X^s}\|v\|^2_{X^s}.$$
\textbf{Step 3:} Now we feed in the term $M_2$,
We want to use local energy norm, so let us cut up the space using
$\chi_j^\alpha(x)$.
%
%
%
%
%
\begin{eqnarray*}
&&\|\lambda \el \chi_j^\alpha(x) \ul \sum_{\gamma\ll\alpha\lesssim\lambda, \sigma\gtrsim\alpha^3}\LL^{-1}\alpha \pq(\va\vr)\|_{\DI}\\
&\lesssim&
\lambda^{3s+\frac{5}{2}}\sum_{\gamma\ll\alpha\lesssim\lambda,
\sigma\gtrsim\alpha^3}\frac{\alpha}{\sigma}\|\el\ul\|_{\fxtt}\|
\el \chi_j^\alpha(x)\pq(\va\vr)\|_{L^2_xL^{\infty}_t}
\\
&\lesssim&\lambda^{3s+\frac{5}{2}}\lambda^{-1-s}\|\ul\|_{X_{}^s}\sum_{\gamma\ll\alpha\lesssim\lambda,
\sigma\gtrsim\alpha^3}\frac{\alpha}{\sigma}\|\el\chi_j^\alpha(x) \pq(\va\vr)\|_{L^2_xL^{\infty}_t}\\
&\lesssim&\sum_{\gamma\ll\alpha\lesssim\lambda,
\sigma\gtrsim\alpha^3}\lambda^{2s+\frac{3}{2}}\|\ul\|_{X_{}^s}
\alpha\sigma^{-1 }\|\el\chi_j^\alpha(x)\va\|_{\txft}\|\el\cax\vr\|_{L^{\infty}_{t,x}}\\
&\lesssim&\sum_{\gamma\ll\alpha\lesssim\lambda,
\sigma\gtrsim\alpha^3}\lambda^{2s+\frac{3}{2}}\alpha^{s+\frac{5}{2}
}
\sigma^{-1 }\gamma \|\ul\|_{X_{}^s}\|\chi_j^\alpha(x)\va\|_{X^1[\iaa]}\|\vr\|_{X_{}^s}\\
&\lesssim&\lambda^{2s+\frac{3}{2}}\alpha^{s+\frac{1}{2}
}\|\ul\|_{X_{}^s}\|\|\chi_j^\alpha(x)\va\|_{X^1[\iaa]}\|\|v\|_{X^s}
\end{eqnarray*}
We can also square sum up the spatial cutoff in the estimate
above, and get \ben \|\lambda \el   \ul M_2\|_{\DI}&\lesssim &
\|\ul\|_{X_{}^s}\|v\|^2_{\xx}. \een In the proof we used the
estimate
\[\|\el\chi_j^\alpha(x)\va\|_{\txft}\lsm \|\chi_j^\alpha(x)\|_{L^4_x}\|\el\chi_j^\alpha(x)\va\|_{L^4_xL^\infty_t}\lesssim\alpha^{s+\frac{3}{2}}\|\ua\|_{\xaa}.\]
Actually we also have $\txft$ maximal function estimate
\cite{KPV3} on small time interval.

We end this section with two bilinear estimates, as a companion to
Proposition~\ref{biliinearprop}. The proof is essentially
repeating what we did preivously.

\begin{prop}\label{L2bilinear}
For $\lambda\gg\alpha$ we have the following estimates
\beq\label{l2xhighs}\|\el\ul
(Q_{\sigma\gtrsim\alpha^3}\va)\|_{\txt}\lesssim
\lambda^{-3s-\frac{5}{2}}\|\ul\|_{\xll}\|v\|^2_{X^s\cap
X^s_{le}},\eeq \beq\label{sumupl2}\|\el\ul\va\|_{\txt}\lesssim
\max{(\lambda^{-1-s}\alpha^{-s},\lambda^{-3s-\frac{5}{2}})}\|\ul\|_{X^s\cap
X^s_{le}}(\|v\|_{X^s\cap X^s_{le}}+\|v\|^2_{X^s\cap
X^s_{le}}),\eeq \beq\label{l2hmodesamefrequency} \|\el
(Q_{\sigma\gtrsim\lambda^3}\ul)\vl\|_{\txt}\lesssim
\lambda^{-3s-\frac{5}{2}}\|\ul\|_{\xx}\|\vl\|_{\xx}.\eeq
\end{prop}
\begin{proof}
For (\ref{l2xhighs}), we reiterate the equation, and notice in all
the proofs we did, we are proving a $L^2$ estimate of the product,
with weight $\lambda^{3s+\frac{5}{2}}$.

For (\ref{sumupl2}), we compare the estimate in the following
cases

If $\ul,\va\in X^1$, we have (\ref{L2XX}); If $\ul\in S,\va\in
X^1$, we have (\ref{L2SSX}).

If $\ul\in X^1, \va\in S$, we have (\ref{l2xhighs}) and \ben
\|\el\ul\va\|_{\txt}\lesssim
\|\el\ul\|_{L^\infty_xL^2_t}\|\el\va\|_{L^2_xL^\infty_t}\lesssim\lambda^{-1-s}\alpha^{-2s-\frac{3}{2}}\|\ul\|_{X^1[\ill]}\|\va\|_{\B[\iaa]}.
\een

If $\ul,\va\in S$, we have (\ref{L2SSS}) except for $\ul\in \B,
\va\in Z$. But notice that the estimate (\ref{l2ulbvaz})  is
larger than $\lambda^{-1-s}\alpha^{-s}$.

Hence we can sum up the estimates to get (\ref{sumupl2}).
\newline
The proof of (\ref{l2hmodesamefrequency}) is carried out in the
same way as all the detailed analysis before. We discuss cases of
$\ul\in X^1$ or $\B$ or $Z$, and be a bit careful when
$Q_{\sigma\gtrsim\lambda^3}\ul\in \B $ or $\K$.
\end{proof}
\section{Energy conservation}
\label{sec5}

In this section, we aim to study the conservation of $H^s$ energy,
this part of calculation follows similar as in~\cite{CKSTT} and
~\cite{KT1}.

Given a positive multiplier $a$, we set
\[E_2(u)=<a(D)u,u>.\] We want to take the symbol
$a(\xi)=(1+\xi^2)^s$, but as in \cite{KT1}, \cite{KT2},we will
allow a slightly larger class of symbols.

\begin{defin}\label{Sa}
 a) Let $s\in \mathbb{R}, \epsilon>0$. Then $S_{\epsilon}^s$ is the class of spherically symmetric symbols with
the following properties:\\
(i) symbol regularity,
\[|\partial^{\alpha}a(\xi)|\lesssim  a(\xi)(1+\xi^2)^{-\frac{\alpha}{2}}.\]
(ii) decay at infinity,
\[s\leq \frac{\ln a(\xi)}{\ln (1+\xi^2)} \leq  s+\epsilon, \mbox{~~~} s-\epsilon\leq \frac{\ln a(\xi)}{\ln (1+\xi^2)} \leq  s+\epsilon.\]
b) If $a$ satisfies (i) and (ii) then we say that $d$ is dominated
by $a$, written as $d \in S(a)$, if
\[|\partial^{\alpha}d|\lesssim a(\xi)(1+\xi^2)^{-\frac{\alpha}{2}},\]
with constant depending only on $a$.
\end{defin}
\begin{defin}
(a) A $k$-multiplier generates a $k$-linear functional or $k$-form
acting on $k$ functions $u_1, \cdots , u_k$
\[\Lambda_k(m;u_1,\cdots,u_k)=\int_{\xi_1+\cdots+\xi_k=0}m(\xi_1,\cdots,\xi_k )\widehat{u}_1(\xi_1)\cdots\widehat{u}_k(\xi_k).\]
We will write $\Lambda_k(m)$ for $\Lambda_k(m;u,\cdots,u)$.\\
(b) The symmetrization of a $k$-multiplier $m$ is the multiplier
\[[m]_{sym}(\xi) =\frac{1}{n!}\sum_{\sigma\in S_k}m(\sigma(\xi)).\]
\end{defin}

We have the following computation ~\cite{CKSTT}.
\begin{prop}
Suppose $u$ satisfies the KdV equation \textnormal{(\ref{kdv})
}and  $m$ is a symmetric $k$-multiplier. Then
\[\frac{d}{dt}\Lambda_k(m) = \Lambda_k(m\Delta_k) -i\frac{k}{2}\Lambda_{k+1}(m(\xi_1,\cdots,\cdots \xi_{k-1},\xi_k+\xi_{k+1})(\xi_k+\xi_{k+1})), \]
where
\[\Delta_k = i(\xi_1^3+\cdots+\xi_k^3).\]
\end{prop}
\subsection{Symbol calculation of modified energy}
Here we construct modified energy, following the calculation   in
\cite{CKSTT}.

We first compute the derivative of $E_2$ along the flow\[\frac{d }{dt}E_2(u)=\Lambda_3(M_3).\] Easy to see that $M_3=c\sum_{i=1}^3(a(\xi_i)\xi_i)$, we will ignore the constant.\\
Now we form modified energy
\[E_3(u)=E_2(u)+\Lambda_3(\sigma_3),\] and
we aim to choose the symmetric 3-multiplier $\sigma_3$ to achieve
a cancellation.
\[\frac{d}{dt}E_3(u)=\Lambda_3(M_3)+\Lambda_3(\sigma_3\Delta_3)+\Lambda_4(-i\frac{3}{2}\sigma_3(\xi_1,\xi_2,\xi_3+\xi_4)(\xi_3+\xi_4)).\]
So if we take  \[\sigma_3=-\frac{M_3}{\Delta_3},\]  we get
\[\frac{d}{dt}E_3(u)=\Lambda_4(M_4), \hspace{.3in}M_4=-i\frac{3}{2}[\sigma_3(\xi_1,\xi_2,\xi_3+\xi_4)(\xi_3+\xi_4)]_{sym}.\]
Similarly, we can define
$E_4(u)=E_3(u)+\Lambda_4(\sigma_4),\mbox{~~}
\sigma_4=-\frac{M_4}{\Delta_4}$,
\[\frac{d}{dt}E_4(u)=\Lambda_5(M_5),\]
then we have
\[M_5=-2i[\sigma_4(\xi_1,\xi_2,\xi_3,\xi_4+\xi_5)(\xi_4+\xi_5)]_{sym}.\]
This process can be continued to have further corrections, but we
will stop here, since higher corrections are harder to estimate.
\subsection{Bounds for multipliers} In order to estimate the derivative of modified energy,
we need to have good bounds for $M_i$ and $\sigma_i$.
 Also now  $M_i$  is defined only on the diagonal $\xi_1+\cdots\xi_k=0$, but in order to separate variables, we want to extend it off diagonal, this is useful when we prove local energy decay later on.
\begin{prop}\label{bc}
Assume that $a\in S_\epsilon^{s}$ and $d\in S(a)$, then there
exist functions $b$ and $c$ such that
\[\sum_{i=1}^{3}a(\xi_i)\xi_i=b(\xi_1, \xi_2, \xi_3)(\xi_1^3+\xi_2^3+\xi_3^3)+c(\xi_1, \xi_2, \xi_3)(\xi_1+\xi_2+\xi_3).\]
And on each dyadic region  $\{\xi_1\sim \alpha, \hspace{0.03in}
\xi_2\sim \lambda,\hspace{0.03in} \xi_3\sim\mu$,
$\alpha\leq\lambda\leq\mu\}$,  we have the regularity conditions
\[\partial_1^{s_1}\partial_2^{s_2}\partial_3^{s_3}b(\xi_1,\xi_2,\xi_3)\lesssim a(\alpha)\lambda^{-1}\mu^{-1}\alpha^{-s_1}\lambda^{-s_2}\mu^{-s_3},\]
\[\partial_1^{s_1}\partial_2^{s_2}\partial_3^{s_3}c(\xi_1,\xi_2,\xi_3)\lesssim a(\alpha)\lambda^{-1}\mu\alpha^{-s_1}\lambda^{-s_2}\mu^{-s_3}.\]
\end{prop}
\begin{proof}  Since \[\xi_1^3+\xi_2^3+\xi_3^3=3\xi_1\xi_2\xi_3+(\xi_1+\xi_2+\xi_3)(\xi_1^2+\xi_2^2+\xi_3^2-\xi_1\xi_2-\xi_2\xi_3-\xi_1\xi_3).\]
Let's construct
\[b=\frac{\sum_{i=1}^{3}a(\xi_i)\xi_i}{3\xi_1\xi_2\xi_3},\]
\[c=-b(\xi_1^2+\xi_2^2+\xi_3^2-\xi_1\xi_2-\xi_2\xi_3-\xi_1\xi_3).\]
Notice that $a(x)x$ is a decreasing function for $x$, then the
estimates are straightforward.
\end{proof}
\subsubsection{Bound for $M_3$ and $\sigma_3$}
We have $M_3=\sum_{i=1}^3a(\xi_i)\xi_i$,
$\sigma_3=\frac{M_3}{\Delta_3}$ modulo a constant.
\begin{prop}\label{m3}
On the set \[\Omega=\{\xi_1+\xi_2+\xi_3=0\} \cap \{\xi_1\sim
\alpha,  \xi_2\sim\xi_3\approx \lambda\geq \alpha\}\] we have
\[|M_3(\xi_1,\xi_2,\xi_3)|\lesssim a(\alpha)\alpha,\]
\[|\sigma_3(\xi_1,\xi_2,\xi_3)|\lesssim \frac{a(\alpha)}{\lambda^2}.\]
\end{prop}
\begin{proof} If $\alpha\approx \lambda$, no need to do any proof. In case $\alpha\ll\lambda$, using the fact $a$ is spherical symmetric,
\[\sum_{i=1}^3a(\xi_i)\xi_i=a(\xi_1)\xi_1-a(\xi_2)\xi_1 -a(\xi_2)\xi_3+a(\xi_3)\xi_3\]
and we have $|a(\xi_3)\xi_3-a(\xi_2)\xi_3|\lesssim
|a'(\xi_3)\xi_1\xi_3|\lesssim|a(\xi_3)\xi_1|$. So the estimate for
$M_3$ become obvious. Using the fact that
$\Delta_3=3\xi_1\xi_2\xi_3$ on set $\Omega$, we get bounds for
$\sigma_3$.
\end{proof}

From this we can prove that $E_3(u)$ is bounded by $E_2(u)$.
\begin{prop}\label{e3e}
We have the fact that \beq
\label{e3ebound}|\Lambda_3(\sigma_3)|\lesssim
|E_2(u)|^{\frac{3}{2}}.\eeq
\end{prop}
\begin{proof} We can expand the trilinear expression in   dyadic frequency band $\{\lambda,\lambda,\alpha\leq \lambda\}$.
Then using the estimate for $\sigma_3$, we can bound
$|\Lambda_3(\sigma_3)|$ by  \ben a(\alpha)\lambda^{-2}\int
\ul\ul\ua dx
&\lesssim & a(\alpha)\lambda^{-2}\alpha^{\frac{1}{2}}\|\ul\|_{L^2}\|\ul\|_{L^2}\|\ua\|_{L^2}\\
&\lesssim&  ( a(\alpha)\alpha)^{\frac{1}{2}}
(a(\lambda)\lambda^2)^{-1}E_2(\ul)E_2(\ua)^{\frac{1}{2}}.\een We
can sum up the frequencies and get (\ref{e3ebound}).
\end{proof}
\subsubsection{Bound for $M_4$ and $\sigma_4$} Recall that \[M_4=-i\frac{3}{2}[\sigma_3(\xi_1,\xi_2,\xi_3+\xi_4)(\xi_3+\xi_4)]_{sym}\] We adopt the calculation done in \cite{CKSTT} (Notice, our $a(\xi)$ corresponds to $m^2(\xi)$ , $\Delta_k$ corresponds to $\alpha_k$ in their paper), we have the following formula for $M_4$
\begin{eqnarray}\label{m4}
M_4(\xi_1,\xi_2,\xi_3,\xi_4)
=\frac{-1}{108}\frac{\Delta_4}{\xi_1\xi_2\xi_3\xi_4}[a(\xi_1)+\cdots+a(\xi_4)-a(\xi_{12})-a(\xi_{13})-a(\xi_{14})]
\end{eqnarray}
\ben
\hspace{-1.2in}+\frac{1}{36}[\frac{a(\xi_1)}{\xi_1}+\cdots+\frac{a(\xi_4)}{\xi_4}].
\een Here we used the notation $\xi_{jk}=\xi_j+\xi_k$, and
\begin{eqnarray}\Delta_4=\xi_1^3+\xi_2^3+\xi_3^3+\xi_4^3
=3(\xi_1\xi_2\xi_3+\xi_1\xi_2\xi_4+\xi_1\xi_3\xi_4+\xi_2\xi_3\xi_4)
=3\xi_{12}\xi_{13}\xi_{14}.\label{detlaformula}
\end{eqnarray}
\begin{prop} \label{m4bound}We have the estimate for $M_4$
\beq|M_4|\lesssim\frac{\Delta_4 a(\min
(|\xi_i|,|\xi_{jk}|))}{|\xi_1\xi_2\xi_3\xi_4|}.\label{}\eeq
\end{prop}
\begin{proof}  The proof repeats the argument of Lemma 4.4 in \cite{CKSTT}. We can also deduce it from our next proposition.
\end{proof}
We have bounds on $\sigma_4$ immediately from
Proposition~\ref{m4bound}. But in order to do correction, we need
improve it slightly.
\begin{prop}
\beq|\sigma_4|\lesssim \frac{ a(\min
(|\xi_i|,|\xi_{jk}|))}{|\xi_1\xi_2\xi_3\xi_4|}, \hspace{0.5in}
 |\Lambda_4(\sigma_4)|\lesssim |E_2(u)|^2.\label{boundform4}\eeq
\end{prop}
\begin{proof}
We look at $\Lambda_4(\sigma_4)$, expand it into dyadic frequency
components, since $\xi_i$ are symmetric, we can assume
$\xi_1\geq\xi_2\geq\xi_3\geq\xi_4$

(1) $\{\xi_1,\xi_2,\xi_3,\xi_4\}=\{\mu,\mu,\lambda,\lambda\},
\mu\gg\lambda$. Then  we have
$\min(\xi_i,\xi_{ij})=\xi_{12}\lesssim\lambda$ and
$|\sigma_4|\lesssim\frac{a(\xi_{12})}{\lambda^2\mu^2}$.
In this case, we can bound $\Lambda_4(\sigma_4)$   by \ben && a(\xi_{12})\lambda^{-2}\mu^{-2}\int \um\um\ul\ul dx\\
&\lesssim& a(\xi_{12})|\xi_{34}|\lambda^{-2}\mu^{-2}\|\um\|_{L^2}\|\um\|_{L^2}\|\ul\|_{L^2}\|\ul\|_{L^2}\\
&\lesssim&
a(\xi_{12})|\xi_{12}|(a(\mu)\mu^2)^{-1}(a(\lambda)\lambda^{2})^{-1}E_2(\um)E_2(\ul).
\een Here notice that  $a(x)x$ is bounded and we can sum up the
frequencies.

(2) $\{\xi_1,\xi_2,\xi_3,\xi_4\}=\{\mu,\mu,\lambda,\alpha\},
\mu\gg\lambda\gg\alpha$. In this case, we have
$\min(\xi_i,\xi_{ij})=\xi_4$, but we need attention with the
estimate here. In fact, with the expression for $M_4$(\ref{m4}),
we can separate the expression of $\sigma_4$ into two parts.

One term looks like
\[-\frac{1}{108}\frac{1}{\xi_1\xi_2\xi_3\xi_4}[a(\xi_1)+a(\xi_2)-a(\xi_{13})-a(\xi_{14})]+\frac{1}{36\Delta_4}[\frac{a(\xi_1)}{\xi_1} +\frac{a(\xi_2)}{\xi_2}]\]
and it is bounded by $\frac{a(\mu)}{\alpha\lambda\mu^2}$.\\
And the other term looks like(if we ignore the constant
$-\frac{1}{108}$ ), \ben
&&\frac{a(\xi_3)+a(\xi_4)-a(\xi_{12})}{\xi_1\xi_2\xi_3\xi_4}-\frac{1}{\xi_{12}\xi_{13}\xi_{14}}[\frac{a(\xi_3)}{\xi_3}+\frac{a(\xi_4)}{\xi_4}]\\
&=&\frac{a(\xi_3)\xi_3\xi_4\xi_{12}+a(\xi_3)\xi_1\xi_2\xi_3+a(\xi_4)\xi_3\xi_4\xi_{12}+a(\xi_4)\xi_1\xi_2\xi_4-a(\xi_{12})\xi_{12}\xi_{13}\xi_{14}}{\xi_1\xi_2\xi_3\xi_4\xi_{12}\xi_{13}\xi_{14}}
\een So it is bounded by $\frac{a(\alpha)}{\lambda^2\mu^2}$. Now
we can bound $\Lambda_4(\sigma_4)$   by \ben
&&a(\alpha)\lambda^{-2}\mu^{-2}\int \um\um\ul\ua dx\\
&\lesssim&a(\alpha)\lambda^{-2}\mu^{-2}\lambda^{\frac{1}{2}}\alpha^{\frac{1}{2}}\|\um\|_{L^2}\|\um\|_{L^2}\|\ul\|_{L^2}\|\ua\|_{L^2}\\
&\lesssim&(a(\alpha)\alpha)^{\frac{1}{2}}(a(\lambda)\lambda^3)^{-\frac{1}{2}}(a(\mu)\mu^2)^{-1}E_2(\um)E_2(\ul)^{\frac{1}{2}}E_2(\ua)^{\frac{1}{2}}.
\een

(3) $\{\xi_1,\xi_2,\xi_3,\xi_4\}=\{\mu,\mu,\mu,\lambda\},
\mu\gg\lambda$. Here $\min(\xi_i,\xi_{ij})=\lambda$, we  can do
same estimate as in  previous case and get
$|\sigma_4|\lesssim\frac{a(\lambda)}{\mu^4}$, we need bound the
expression \ben
&&a(\lambda)\mu^{-4}\int \um\um\um\ul dx\\
&\lesssim&a(\lambda)\mu^{-4}\lambda^{\frac{1}{2}}\mu^{\frac{1}{2}}\|\um\|_{L^2}\|\um\|_{L^2}\|\ul\|_{L^2}\|\ua\|_{L^2}\\
&\lesssim&(a(\lambda)\lambda)^{\frac{1}{2}}(a(\mu)\mu^2)^{-1}(a(\mu)\mu^3)^{-\frac{1}{2}}
E_2(\um)^{\frac{3}{2}}E_2(\ul)^{\frac{1}{2}}. \een

(4) $\{\xi_1,\xi_2,\xi_3,\xi_4\}=\{\mu,\mu,\mu,\mu\},
\min(\xi_i,\xi_{ij})=\xi_{ij}$. For convenience, suppose it is
$\xi_{12}$, then we have
$|\sigma_4|\lesssim\frac{a(\xi_{12})}{\mu^4}$. And we can bound
$\Lambda_4(\sigma_4)$ by \ben a(\xi_{12})\mu^{-4}\int \um\um\um\um
dx \lesssim a(\xi_{12})|\xi_{12}|(a(\mu)\mu^2)^{-2}E_2(\um)^2.
\een

In all the cases above, we can sum up the frequency and get
(\ref{boundform4}).
\end{proof}
\begin{remark} From the estimate in the proof, we see that
actually we have slightly better bound for $M_4$ than
Proposition~\ref{m4bound}  in the following two cases
\begin{enumerate} \item
$\{\xi_1,\xi_2,\xi_3,\xi_4\}=\{\mu,\mu,\lambda,\alpha\} ,
\alpha\ll\lambda\ll\mu$, $|M_4|\lsm\frac{a(\alpha)}{\lambda}$,\\
\item $\{\xi_1,\xi_2,\xi_3,\xi_4\}=\{\mu,\mu,\mu,\lambda\} ,
 \lambda\ll\mu$, $|M_4|\lsm\frac{a(\lambda)}{\mu}$.\\
\end{enumerate}
\end{remark}
%
%
%
\begin{prop}\label{m4est} We have the error estimate when $s\geq \so$
\[|\int_0^1 \Lambda_4(M_4)  dt|\lsm \|u\|_{\xx}^4(1+\|u\|_{\xx}+\|u\|^2_{\xx})\]
\end{prop}
\begin{proof}
As before, we expand the error term $\Lambda_4(M_4)$ in the dyadic
frequency component and discuss in each cases. Since $u\in \xx$,
we still decompose each piece as $\ul=u_{\lambda,1}+u_{\lambda,2},
\hspace{0.07in}u_{\lambda,1}\in X^1[\ill],
\hspace{0.07in}u_{\lambda,2}\in S[\ill].$ We abuse the notation
and still use $\ul$ to represent any of them. We   assume
$\xi_1\geq\xi_2\geq\xi_3\geq\xi_4$.

One thing to notice the the high modulation relation. Since
 \ben \int_0^1 \Lambda_4(M_4)  dt= \int_{\Sigma} M_4 \widetilde{u}_1\widetilde{u}_2\widetilde{u}_3\widetilde{u}_4\,d\xi\,d\tau.\een
\ben \Sigma=\{\xi_1+\xi_2+\xi_3+\xi_4=0,
 \tau_1+\tau_2+\tau_3+\tau_4=0\}.\een
 We have
 \beq
(\tau_1-\xi_1^3)+(\tau_2-\xi_2^3)+(\tau_3-\xi_3^3)+(\tau_4-\xi_4^3)=-\Delta_4=-3\xii{12}\xii{13}\xii{14}.
 \eeq
Hence we get the high modulation \beq
\label{himodem4}\sM=\max{\{|\tau_1-\xi_1^3|, |\tau_2-\xi_2^3|,
|\tau_3-\xi_3^3|, |\tau_4-\xi_4^3|\}}\gtrsim
|\xii{12}\xii{13}\xii{14}|. \eeq

\noindent\textbf{(1)}
$\{\xi_1,\xi_2,\xi_3,\xi_4\}=\{\mu,\mu,\lambda,\lambda\},
\mu\gg\lambda$. Then  we have
$\min(\xi_i,\xi_{ij})=\xi_{12}\lesssim\lambda$  and $|M_4|\lesssim
|\frac{\ax{12}\xii{12}}{\lambda^2}|$, also notice function $a(x)x$
is bounded. Let us use the crude bilinear estimate
(\ref{sumupl2}), and also we need cut the time interval $[0,1]$
into smaller scale of size $\mu^{4s+3}$.
 \ben
&&\int_0^1\Lambda_4(M_4)dt\\
&\lsm&|\frac{\ax{12}\xii{12}}{\lambda^2}|
(\max{\{\mu^{-1-s}\lambda^{-s},\mu^{-3s-\frac{5}{2}}\}})^2\mu^{-4s-3}\|\um\|_{\xx}^2(\|u\|_{\xx}+\|u\|^2_{\xx})^2\\
&\lsm&\max{\{\mu^{-6s-5}\lambda^{-2s-2},\mu^{-10s-8}\lambda^{-2}\}}
 \|\um\|^2_{\xx}\sum_{k=2}^4\|u\|^k_{\xx}.
 \een
It is summable when $s\geq-\frac{4}{5}$.

\noindent\textbf{(2)}
$\{\xi_1,\xi_2,\xi_3,\xi_4\}=\{\mu,\mu,\lambda,\alpha\},
\mu\gg\lambda\gg\alpha$. $|M_4|\lesssim\frac{a(\alpha)}{\lambda}$.
We estimate it in exactly the same way as (1).
 \ben
&&\int_0^1\Lambda_4(M_4)dt\\
&\lsm&\frac{a(\alpha)}{\lambda}\max{\{\mu^{-1-s}\lambda^{-s},\mu^{-3s-\frac{5}{2}}\}}
\max{\{\mu^{-1-s}\alpha^{-s},\mu^{-3s-\frac{5}{2}}\}}\mu^{-4s-3}
\\
&&\times\|\um\|_{\xx}^2(\|u\|_{\xx}+\|u\|^2_{\xx})^2.
 \een
By computing the exponents, we can sum up the frequencies when
when $s\geq-\frac{4}{5}$.

\noindent \textbf{(3)}
$\{\xi_1,\xi_2,\xi_3,\xi_4\}=\{\mu,\mu,\mu,\lambda\},
\mu\gg\lambda$, here $\min(\xi_i,\xi_{ij})=\lambda$,
$\sm\gtrsim\mu^3$.
\newline \textbf{Case 1} When at least one of $\um$ have high
modulation, here we cut the interval to size $\mu^{4s+3}$ and use
bilinear  on $(\Qm\um)\um$ (\ref{l2hmodesamefrequency}) and
$\um\ul$, we see that we get the bound \ben
&&\int_0^1\int_{\mathbb{R}}\frac{a(\lambda)}{\mu}(\Qm\um)\um\um\ul
dx dt
\\
&\lsm&
\frac{a(\lambda)}{\mu}\mu^{-3s-\frac{5}{2}}\max(\mu^{-1-s}\lambda^{-s},\mu^{-3s-\frac{5}{2}})\mu^{-4s-3}\|\um\|^3_{\xx}(\|u\|_{\xx}+\|u\|_{\xx}^2)\\
&\lsm&\mu^{-10s-9}\|\um\|^3_{\xx}(\|u\|_{\xx}+\|u\|_{\xx}^2).
 \een
So it is summable for $s\geq -\frac{9}{10}$
\newline \textbf{Case 2} When the high modulation fall on $\ul$,
this is the hard case, we use the $L^2$ on $\Qm\ul$, and $L^2$ on
the product $\um\um\um$. \beq\label{l2Xhighmode}\|\el
Q_{\sigma\gtrsim\mu^3}\ul\|_{\txt}\lesssim
\lambda^{-3s-\frac{3}{2}}\mu^{-3}\|\ul\|_{X_\lambda[\ill]},\eeq
\beq\label{l2Shighmode}\|\el
Q_{\sigma\gtrsim\mu^3}\ul\|_{\txt}\lesssim
\lambda^{3+4s}\mu^{-6s-6}\|\ul\|_{\C[\ill]},\eeq
\beq\label{triple}\|\eta_{\mu}\um\um\um\|_{\txt}\lesssim
\mu^{-\frac{1}{2}-3s}\|\um\|_{X^s}^3.\eeq The third one is proved
by discussing $\um\in X^1$ or $S$, and notice that none of them
has high modulation. Then we get \ben
&&\int_0^1\int_{\mathbb{R}}\frac{a(\lambda)}{\mu}\um\um\um  Q_{\sigma\gtrsim\mu^3}\ul dtdx\\
&\lsm&\frac{a(\lambda)}{\mu}\max{\{\lambda^{-3s-\frac{3}{2}}\mu^{-3},\lambda^{3+4s}\mu^{-6s-6}\}}\mu^{-\frac{1}{2}-3s}\mu^{-4s-3}\|\um\|_{\xx}^3\|\vl\|_{\xx}
\een And we can sum up frequencies when $s\geq-\frac{21}{26}$.

\noindent\textbf{(4)}
$\{\xi_1,\xi_2,\xi_3,\xi_4\}=\{\mu,\mu,\mu,\mu\}$ here we need to
discuss the size of $\xi_{ij}$.
\[\xii{12}+\xii{13}+\xii{14}=2\xii{1}\]so at least one of them
is of size $\mu$
\newline \textbf{Case 1:} When $\xii{ij}\gtrsim \mu$, then we have
$|M_4|\lsm\frac{a(\mu)}{\mu}$, and we have the high modulation
factor $\sm\gtrsim\mu^3$, so we use bilinear on $(\Qm\um)\um$, and
$L^2$ for each of $\um\um$.

Notice the (8,4) is Strichartz pair and using the size of interval
we get \beq \|\eta_{\mu}\um\um\|_{\txt}\lesssim
\mu^{\frac{1}{2}-s}\|\um\|_{X^1_{\mu}[\imm]}^2.\eeq

From (\ref{L2SSX}) we have \beq
\label{l2exs}\|\eta_{\mu}\um\um\|_{\txt}\lesssim\mu^{-\frac{3}{2}-2s}
\|\um\|_{X^1_{\mu}[\imm]}\|\um\|_{S[\imm]}.\eeq

From (\ref{L2SSS}) and (\ref{l2ulbvaz}) we get \beq\label{l2esss}
\|\eta_{\mu}\um\um\|_{\txt}\lesssim\mu^{-1-s}
\|\um\|_{S[\imm]}\|\um\|_{S[\imm]}.\eeq

 \ben &&
\int_0^1\int_{\mathbb{R}}\frac{a(\mu)}{\mu}(\Qm\um)\um \um\um
dx dt\\
&\lsm&\frac{a(\mu)}{\mu}\mu^{-3s-\frac{5}{2}}\mu^{\frac{1}{2}-s}\mu^{-4s-3}
\|\um\|^4_{X^s}
\\
&\lsm&a(\mu)\mu^{-8s-6}\|\um\|^4_{X^s} \een so it is summable when
$s\geq -1$.
\newline
\textbf{Case 2:} When two of $\xii{ij}$ is big, one is small,
let's assume $\xii{13}\ll\mu, \xii{12},\xii{14}\gtrsim\mu$, we
have $|M_4|\lsm|\frac{a(\xii{13})\xii{13}}{\mu^2}|$. Then we can
easily calculate that
\[(\xi_1-\xi_2)+(\xi_1-\xi_4)-(\xi_1+\xi_3)=2\xi_1\] since
$\xii{13}\ll\mu$, we must have at least one of $\xi_1-\xi_2$ or
$\xi_1-\xi_4$ be of size $\mu$, with out loss of generality, we
assume $|\xi_1-\xi_2|\gtrsim \mu$, so we have separation of
frequency, i.e \[|\xi_1-\xi_2|\approx\mu, |\xi_1+\xi_2|\approx
\mu\] and we can also prove that
\[|\xi_3+\xi_4|=|\xi_1+\xi_2|\approx \mu,
|\xi_3-\xi_4|=|\xi_3+\xi_1-(\xi_1+\xi_4)|\approx\mu\]
Now we have the bilinear estimate of two $\um$'s which have
frequency separation. \beq\label{l2exxx}
\|\eta_{\mu}\um\um\|_{\txt}\lesssim\mu^{-1-2s}\|\um\|_{\xmm}. \eeq
Together with (\ref{l2exs}) and (\ref{l2esss}), we get \ben &&
\int_0^1\int_{\mathbb{R}}|\frac{a(\xii{12})\xii{12}}{\mu^2}|\um\um
\um\um dx
dt\\
&\lsm&
|\frac{a(\xii{12})\xii{12}}{\mu^2}|(\mu^{-1-2s})^2\mu^{-4s-3}
\|\um\|^4_{X^s}
\\
&\lsm&\mu^{-8s-7} \|\um\|^4_{X^s}.\een so we can sum up for $s\geq
-\frac{7}{8}$.

\noindent\textbf{Case 3:} When one of $\xii{1j}$ is big, the other
two small. We can assume $\xii{12}\leq \xii{13}\ll\mu$,
$\xii{14}\gtrsim\mu$. In this case, we don't have frequency
separation.
$|M_4|\lsm|\frac{a(\xi_{12})\xi_{12}\xi_{13}}{\mu^3}|$.

But we still have (\ref{L2XXHHL}), so together with (\ref{l2exs})
and (\ref{l2esss}), we get \ben
&&\int_0^1\int{\mathbb{R}}|\frac{a(\xi_{12})\xi_{12}\xi_{13}}{\mu^3}|\um\um\um\um
dx dt\\
&\lsm&
|\frac{a(\xi_{12})\xi_{12}\xi_{13}}{\mu^3}|(|\xi_{13}|^{-\frac{1}{2}}\mu^{-\frac{1}{2}})^2\mu^{-4s}\mu^{-4s-3}\|\um\|^4_{X^s}\\
&\lsm&\mu^{-8s-7}\|\um\|^4_{X^s} \een so it is summable when $s
\geq -\frac{7}{8}$.
\end{proof}

\section{Local energy decay}
\label{sec6}
 Let  $\chi(x)$  be a  positive, rapidly decaying function,  with Fourier transform supported in [-1, 1]. Let $a$
be as in the previous section. We define the indefinite quadratic
form
\beq\widetilde{E}_2(u)=\sum_{\lambda}\frac{1}{2}\int(\phi_\lambda
\ta(D)+ \ta(D)\phi_\lambda)\ul \ul
dx.\label{definelocalenergy}\eeq Here $\phi_\lambda$ is an odd
smooth function whose derivative has the form
$\phi_\lambda'(x)=\psi_\lambda(x)^2$,
$\psi_\lambda(x)=\lambda^{-2s-\frac{5}{2}}\chi(\frac{x}{\lambda^{4s+5}})$.
We will abuse the notation a bit,  and  (\ref{definelocalenergy})
\beq\label{redefinelocalenergy}\widetilde{E}_2(u)=
\frac{1}{2}\int(\phi \ta(D)+ \ta(D)\phi )u u dx,\eeq with the
understanding that it is really defined on each dyadic pieces, and
$\phi=\phi_\lambda$ on each piece.

 Then we have the calculation
\beq\label{derivativeoflocalenergy}\frac{d}{dt}\widetilde{E}_2(u)=\tr_2(u)+\tr_3(u),\eeq
where
\[\tr_2(u)=\langle(\ta(D)\phi_x+\phi_x\ta(D))u_x,u_x\rangle+\langle(\ta(D)\phi_{xxx}+\phi_{xxx}\ta(D))u,u\rangle,\]
\[\tr_3(u)=c Re\langle (\ta(D)\phi+\phi\ta(D))u, (u^2)_x\rangle.\]
We will see in the following propositions that $\tr_2$ can be used
to measure local energy.
\begin{prop}\label{EE}
Let $a\in S^s_{\epsilon}$, $\phi$  defined as above, then we have
the fixed time bound
\[|\widetilde{E}_2(u)|\lesssim E_2(u),\]
\[|\langle(\ta(D)\phi_{xxx}+\phi_{xxx}\ta(D))u,u\rangle|\lesssim E_2(u).\]
\end{prop}
\begin{proof} Since $\phi$ and $\phi_{xxx}$ are bounded and its fourier transform has compact support,
\[|\langle\ta(D)\phi u,u\rangle|=|\langle (\ta(D)^{1/2}\phi\ta(D)^{-1/2}) \ta(D)^{1/2}u, \ta(D)^{1/2}u\rangle|\lesssim E_2(u).\]
Other terms are proved similarly.
\end{proof}
\begin{prop}\label{localengy} We can use $R_2$ to bound the local energy
\beq\|\psi\ta(D)^{\frac{1}{2}}D u\|_{L^2_x}^2\lesssim \tr_2(u)+c
E_2(u).\label{boundr2}\eeq
\end{prop}
\begin{proof} \[\langle(\ta(D)\phi_x+\phi_x\ta(D))u_x,u_x\rangle=2\|\psi(\ta(D)^{\frac{1}{2}}D)u\|_{L^2_x}^2+\langle r^{w}(x,D)u,u\rangle.\]
Here \[r^{w}(x,D)=[\ta(D)^{1/2},[\ta(D)^{1/2},\psi^2]],\] so its
symbol $r$ satisfy the estimate
\[\partial^{\alpha}_x\partial^{\beta}_{\xi}r(x,\xi)\lesssim \lag x\rag^{-N}(1+\xi)^{-\frac{\beta}{2}}\ta(\xi).\]
Hence
\[|\lag r^w(x,D)u,u\rag|\lesssim E_2(u).\]
Combine with previous proposition and the formula for $\tr_2$, we
get   estimate (\ref{boundr2}).
\end{proof}
Integrating (\ref{derivativeoflocalenergy}) and (\ref{boundr2}) on
time interval $[0,1]$, together with Proposition~\ref{EE} we get
\beq\int_0^1\|\psi\ta(D)^{\frac{1}{2}}D u\|_{L^2_x}^2dt \lsm
\|u\|_{\lh}^2 +|\int_0^1\tr_3(u)dt|.\eeq

Next,  we can rewrite $\tr_3$ in the Fourier space. Notice that
original definition of (\ref{definelocalenergy}) is on dyadic
pieces, so $\tr_3$ takes the following form
\[\tr_3(u)=2\int_{\mathbb{R}}\phi(x)e^{ix\xi}\int_{P_\xi}(\ta(\xi_1-\xi)+\ta(\xi_1))\chi(\xi)(\xi_{23})\hu(\xi_1)\hu(\xi_2)\hu(\xi_3)d\xi_id\xi dx,\]
\[P_\xi=\{\xi_1+\xi_2+\xi_3=\xi\}.\]
Here $\phi$ is actually $\phi_\lambda$, $\chi(\xi) $ is the
multiplier used to define projection $P_\lambda$.

Now we can symmetrize it, using the notation
$A(\xi_i)=(\ta(\xi_i-\xi)+\ta(\xi_i))\chi(\xi_i)$
\begin{eqnarray*}
\tr_3(u)&=&\int_{\mathbb{R}}\phi(x)e^{ix\xi}\int_{P_\xi}(\sum_{i=1}^3A(\xi_i))\xi\hu(\xi_1)\hu(\xi_2)\hu(\xi_3)d\xi_id\xi dx\\
&-&\int_{\mathbb{R}}\phi(x)e^{ix\xi}\int_{P_\xi}(\sum_{i=1}^3A(\xi_i)\xi_i)\hu(\xi_1)\hu(\xi_2)\hu(\xi_3)d\xi_id\xi
dx.
\end{eqnarray*}
To better estimate it, we use proposition~\ref{bc}, and rewrite
\beq\label{BCdecomposition}\sum_{i=1}^3A(\xi_i)\xi_i=B(\xi_1,
\xi_2, \xi_3)(\xi_1^3+\xi_2^3+\xi_3^3)+C(\xi_1, \xi_2,
\xi_3)(\xi_1+\xi_2+\xi_3) .\eeq So we split $\tr_3$ into
\[\tr_3(u)=\tilde{R}_{good,3}+\tr_{bad,3},\]
where $\tilde{R}_{good,3}$ and $\tr_{bad,3}$ take the following
form,
\[\tr_{good,3}=\int_{\mathbb{R}}\phi(x)e^{ix\xi}\int_{P_\xi}(\sum_{i=1}^3A(\xi_i)-C)\xi \hspace{0.04in}\hu(\xi_1)\hu(\xi_2)\hu(\xi_3)d\xi_id\xi dx,\]
\[\tr_{bad,3}=-\int_{\mathbb{R}}\phi(x)e^{ix\xi}\int_{P_\xi}(B(\xi_1,\xi_2,\xi_3)(\xi_1^3+\xi_2^3+\xi_3^3))\hu(\xi_1)\hu(\xi_2)\hu(\xi_3)d\xi_id\xi dx.\]
\begin{prop}
Let $a, \phi$ as before, then we have the estimate
\[|\int_0^1\tr_{good,3}(u)dt|\lesssim \sum_{k=3,4}\|u\|^k_{X^s\cap X^s_{le}}.\]
\end{prop}
\begin{proof} As in proposition~\ref{bc}, we have
\[C=-\frac{\sum A(\xi_i)\xi_i}{3\xi_1\xi_2\xi_3}(\xi_1^2+\xi_2^2+\xi_3^2-\xi_1\xi_2-\xi_1\xi_3-\xi_2\xi_3).\]
Let's look at one term of  $\sum A(\xi)-C$,
\begin{eqnarray*}
&&A(\xi_1)+\frac{A(\xi_1)\xi_1}{3\xi_1\xi_2\xi_3}(\xi_1^2+\xi_2^2+\xi_3^2-\xi_1\xi_2-\xi_1\xi_3-\xi_2\xi_3)\\
&=&
\frac{A(\xi_1)}{3\xi_2\xi_3}[(\xi_1+\xi_2+\xi_3)^2-3\xi_1(\xi_1+\xi_2+\xi_3)+3\xi_1^2].
\end{eqnarray*}
So on $P_\xi$, we have
\begin{eqnarray*}&& \sum A(\xi)-C\\
&=&\frac{\sum A(\xi_i)\xi_i\xi^2-3\sum A(\xi_i)\xi_i^2\xi+3\sum
A(\xi_i)\xi_i^3}{3\xi_1\xi_2\xi_3}.\end{eqnarray*} When we feed it
to the integral, we can   do integration by parts to trade $\xi$
for derivative of $\phi$
\begin{eqnarray*}
\tr_{good,3}&=&-i\int_{\mathbb{R}}\phi_{xxx}(x)e^{ix\xi}\int_{P_\xi}\frac{\sum
A(\xi_i)\xi_i}{3
\xi_1\xi_2\xi_3 } \hu(\xi_1)\hu(\xi_2)\hu(\xi_3)d\xi_id\xi dx\\
&+&\int_{\mathbb{R}}\phi_{xx}(x)e^{ix\xi}\int_{P_\xi}\frac{\sum
A(\xi_i)\xi_i^2}{
\xi_1\xi_2\xi_3 }\hu(\xi_1)\hu(\xi_2)\hu(\xi_3)d\xi_id\xi dx\\
&+&i\int_{\mathbb{R}}\phi_{x}(x)e^{ix\xi}\int_{P_\xi}\frac{\sum
A(\xi_i)\xi_i^3}{ \xi_1\xi_2\xi_3
}\hu(\xi_1)\hu(\xi_2)\hu(\xi_3)d\xi_id\xi dx.
\end{eqnarray*}
Let's decompose the region into dyadic region
$\{\alpha,\lambda,\lambda\}$, $\alpha\leq \lambda$ and we can
estimate the symbols, using the fact $a\in S^s_\epsilon$, the the
proof is similar to proposition~\ref{m3}.
\[|\frac{\sum A(\xi_i)\xi_i}{3
\xi_1\xi_2\xi_3 }|\lesssim \frac{a(\alpha)}{\lambda^2},
\hspace{0.2in}|\frac{\sum A(\xi_i)\xi^2_i}{3 \xi_1\xi_2\xi_3
}|\lesssim \frac{a(\lambda)}{\alpha}, \hspace{0.2in}|\frac{\sum
A(\xi_i)\xi^3_i}{3 \xi_1\xi_2\xi_3 }|\lesssim
\frac{a(\lambda)\lambda}{\alpha}.\] The three terms in
$\tr_{good,3}$ are similar, so we only do the third term, since
that has the worst bound. Denote it as $III$
\begin{eqnarray*}
|\int_0^1 III| &\lesssim&
\frac{a(\lambda)\lambda}{\alpha}\int_0^1\int_{\mathbb{R}}\phi_x(x)\ul\ul\ua
dx dt.
\end{eqnarray*}
\textbf{Case 1.} $\alpha\ll\lambda$, put $L^2$ on one of $\ul$,
and bilinear estimate on $\ul\ua$ (\ref{sumupl2}) , also notice
$\phi_x$ is fast decaying on spatial scale $\lambda^{4s+5}$, so we
can use local energy norm to avoid interval summation. ( Based on
our computation below, we can even perform interval summation with
no difficulty.)

\begin{eqnarray*}
 |\int_0^1 III|
&\lesssim& \frac{a(\lambda)\lambda}{\alpha}\sum_{I_\lambda}\int_{I_\lambda}\int_{\mathbb{R}}\phi_x(x)\ul\ul\ua dx dt\\
&\lesssim& \sum_{I_\lambda}\frac{a(\lambda)\lambda}{\alpha}\lambda^{-4s-5}\|\el \chi^\lambda(x)\ul\|_{\txt}\|\el\clx\ul\ua\|_{\txt}\\
&\lesssim&
\frac{a(\lambda)\lambda^{-4s-4}}{\alpha}\lambda^{\frac{3}{2}+s}\max\{\lambda^{-3s-\frac{5}{2}},
\lambda^{-1-s}\alpha^{-s}\}\sum_{I_\lambda}\|\el\clx\ul\|_{\xl}^2
(\|u\|_{X^s}+\|u\|_{X^s}^2)\\
&\lesssim& \lambda^{-4s-5}\max \{\alpha^{-1},
\lambda^{2s+\frac{3}{2}}\alpha^{-1-s}\}\|\ul\|^2_{X^s_{le}}(\|u\|_{X^s}+\|u\|^2_{X^s}).
\end{eqnarray*}
\textbf{Case 2.} $\alpha\approx\lambda$, notice we have high
modulation $\sm\gtrsim \lambda^3$. Then bound $(\Qm \ul  )\ul$ in
$L^2$ (\ref{l2hmodesamefrequency}), and the other one in $L^2$.
\begin{eqnarray*}
 |\int_0^1 III|
&\lesssim& a(\lambda)\sum_{I_\lambda}\int_{I_\lambda}\int_{\mathbb{R}}\phi_x(x)(\Qm\ul)\ul  \ul dx dt\\
&\lesssim& a(\lambda)\lambda^{-3s-\frac{5}{2}}\lambda^{\frac{3}{2}+s}\lambda^{-4s-5}\|\ul\|_{X^s\cap X^s_{le}}^3\\
&\lesssim&\lambda^{-4s-6}\|\ul\|_{X^s\cap X^s_{le}}^3.
\end{eqnarray*}
\end{proof}
For the part $\tr_{bad,3}$ we can not estimate it directly, so we
will add some correction as we did before. Take
\[\td{E}_3(u)=\td{E}_2(u)+ \Lambda_{B}(u),\]
\[\Lambda_{B}(u)=-i\int_{\mathbb{R}}\phi(x)e^{ix\xi}\int_{P_\xi}B(\xi_1,\xi_2,\xi_3)\hu(\xi_1)\hu(\xi_2)\hu(\xi_3)d\xi_id\xi dx.\]
Notice (\ref{BCdecomposition}), then we have
\[\ddt\td{E}_3(u)=\tr_2(u)+\tr_{good,3}+\tr_4(u).\] Here
\[\tr_4(u)=-\int_{\mathbb{R}}\phi(x)e^{ix\xi}\int_{P_{\xi}}[B(\xi_1,\xi_2,\xi_{34})\xi_{34}]_{sym}\hu({\xi_1})
\hu({\xi_2})\hu({\xi_3})\hu({\xi_4})d\xi_id\xi dx.\] Here we need
to do two things, show $|\td{E}_3(u)|\lesssim
E^{\frac{3}{2}}_2(u)$, which is same as proposition \ref{EE} and
\ref{e3e} and \ref{localengy} . And estimate the error, which
repeats the proof of proposition~\ref{m4est} using the fact that
$\phi$ is bounded, $\hat{\phi}$ has compact support, so they does
no change to the proof, in fact, since we have the spatial
localization, we have the privilege of omitting the interval
summation by control them in local energy space.

\begin{prop} With $a, \phi$ as before, when $s\geq \so$, we have the error estimate
\[|\int_0^1\tr_4(u) dt|\lsm \|u\|_{\xx}^4(1+\|u\|_{\xx}+\|u\|^2_{\xx}).\]
\end{prop}

Combining all the propositions in this section, we get the local
energy bound.
\begin{lemma} The solution to the KdV
equation (\ref{kdv}) satisfy the following bound \ben
&&\sum_{\lambda}\lambda^{-2s-5}\sup_j\|\chi_j^{\lambda}\partial_x
u_{\lambda}\|^2_{L^2_{x,t}}
\\
&\lesssim&
\sup_{t}\|u(t)\|^2_{H^s}(1+\|u(t)\|_{H^s})+\|u\|^3_{X^s\cap
X^s_{le}}+\|u\|_{\xx}^4(1+\|u\|_{\xx}+\|u\|^2_{\xx}).\een
\end{lemma}

\section{Finishing the proof}
To finish the whole argument, we  need to pick suitable symbol
$a(\xi)$ in the previous two sections.  As in \cite{KT1}, we pick
slow varying sequence.
\[\beta_{\lambda}^0=\frac{\lambda^{2s}\|u_{0\lambda}\|^2_{L^2}}{\|u_0\|^2_{H^s}},\]
\[\beta_\lambda=\sum_{\mu}2^{-\frac{\epsilon}{2}|\log\lambda-\log\mu|}\beta_{\mu}^0.\]
These $\beta_\lambda$ satisfy the following property
\begin{enumerate}[(i)]
\item $\lambda^{2s }\|u_{0\lambda}\|^2_{L^2}\lsm\beta_{\lambda}
\|u_0\|^2_{H^s},$\\ \item $\sum\beta_\lambda\lsm 1$,\\
\item $\beta_\lambda$ is slow varying in the sense that
\beq\label{slowvarying}|\log_2{\beta_\lambda}-\log_2{\beta_\mu}|\lsm
\frac{\epsilon}{2}|\log_2\lambda-\log_2\mu|.\eeq
\end{enumerate}Now if we take
$a_\lambda=\lambda^{2s}\max(1,\beta_{\lambda_0}^{-1}2^{-\epsilon|\log_2\lambda-\log_2{\lambda_0}|})$,
and correspondingly we take \[a(\xi)\approx
a_\lambda,\hspace{0.2in} |\xi|\approx\lambda\] Then from the slow
varying property (\ref{slowvarying}), we get
\[\sum_{\lambda}a_\lambda\|u_{0\lambda}\|_{L^2_x}^2\lsm \sum_{\lambda}\lambda^{2s}\|u_{0\lambda}\|_{L^2_x}^2
+2^{-\epsilon|\log_2\lambda-\log_2{\lambda_0}|}
\lambda^{2s}\beta_{\lambda_0}^{-1}\|u_{0\lambda}\|_{L^2_x}^2\lsm
\|u_0\|^2_{H^s}.\] Assume that $\|u\|_{\lh}\ll 1,$ which implies
$\sup_tE_2(u(t))\ll 1$. Recall that
\[\frac{d}{dt}(E_2(u)+\Lambda_3(\sigma_3))=\Lambda_4(M_4),\]
so from  Proposition \ref{e3e} and \ref{m4est}, we get
\[(\sum_{\lambda}a(\lambda)\|u_\lambda(t)\|_{L^2_x}^2)^\frac{1}{2}\lsm \|u_0\|_{H^s}+\|u\|_{\xx}^4(1+\|u\|_{\xx}+\|u\|^2_{\xx}).\]
 At fixed frequency $\lambda=\lambda_0$, we get \[\sup_{t}\lambda_0^s\|u_{\lambda_0}(t)\|_{L^2}\lsm \beta_{\lambda_0}^{\frac{1}{2}}(\|u_0\|_{H^s}+\|u\|_{\xx}^4(1+\|u\|_{\xx}+\|u\|^2_{\xx})).\]
From the property of $\beta_\lambda$, we can sum up $\lambda_0$,
and get (\ref{energybound}).

Together with the previous section, we can prove the local energy
bound in exactly the same way, so we conclude the proof of
proposition~\ref{energypropagation}.

\vspace{.2in} \textbf{Acknowledgement}. The author would like to
thank his   advisor, Daniel Tataru,  for suggesting the problem
and for all the guidance and encouragement along the way.

\end{document}